\documentclass[12pt, letterpaper]{article}

\usepackage{color}
\usepackage{setspace}           
\usepackage{graphicx}           
\usepackage{amsmath, bbm, amsfonts, amsthm}            
\usepackage{marginnote}         
\usepackage{datetime}           
\usdate                         
\usepackage{enumitem}           
\usepackage{subfigure}          
\usepackage{rotating}           
\usepackage{hyperref}           
\usepackage{float}              
\usepackage[square,sort,comma,numbers]{natbib}

\usepackage{texintro}           
\usepackage{tocvsec2}           
\usepackage{bibentry}           
\usepackage{breqn}
\usepackage{times}
\usepackage{fancyhdr,graphicx,amsmath,amssymb}
\usepackage[ruled,vlined]{algorithm2e}
\include{pythonlisting}

\SetKwInput{KwData}{Inputs}
\SetKwInput{KwResult}{Outputs}

\usepackage[margin=1in]{geometry}

\makeatletter\let\chapter\@undefined\makeatother 

\newcommand{\E}{\mathbb{E}}
\newcommand{\cP}{\mathcal{P}}
\newcommand{\M}{\mathbb{M}}
\newcommand{\NN}{\mathbb{N} \mathbb{N}}
\newcommand{\LSTM}{\mathbb{L} \mathbb{S} \mathbb{T} \mathbb{M}}
\newcommand{\N}{\mathbb{N}}
\newcommand{\Z}{\mathbb{Z}}
\newcommand{\R}{\mathbb{R}}

\renewcommand{\P}{\mathbb{P}}
\newcommand{\e}{\epsilon}

\newcommand{\F}{\mathcal{F}}
\newcommand{\A}{\mathbb{A}}

\newcommand{\x}{\mathbf{x}}
\newcommand{\X}{\mathbf{X}}

\newcommand{\1}{\mathbbm{1}}
\newcommand{\W}{\Omega}

\renewcommand{\a}{\boldsymbol{\alpha}}

\renewcommand{\L}{\mathcal{L}}

\theoremstyle{plain}
\newtheorem{Theorem}{Theorem}[section]
\newtheorem{Lemma}[Theorem]{Lemma}

\theoremstyle{definition}

\theoremstyle{remark}
\newtheorem{Remark}[Theorem]{Remark}

\numberwithin{equation}{section}

\numberwithin{figure}{section}

\begin{document}

\nobibliography*                                
\setlist{noitemsep}                             


\title{Deep Learning Methods for Mean Field Control Problems with Delay}

\author{
Jean-Pierre Fouque
    \thanks{Department of Statistics \& Applied Probability, University of California, Santa Barbara, CA \ 93106-3110, e-mail: fouque@pstat.ucsb.edu. Work  supported by NSF grant DMS-1814091.}
\and
Zhaoyu Zhang 
    \thanks{Department of Statistics \& Applied Probability, University of California, Santa Barbara, CA \ 93106-3110, e-mail: zhaoyu\_zhang@ucsb.edu}
}

\date{\today}


\maketitle

\begin{abstract}
We consider a general class of mean field control problems described by stochastic delayed differential equations of McKean-Vlasov type.  Two numerical algorithms are provided  based on deep learning techniques, one is to directly parameterize the optimal control using neural networks, the other is based on numerically solving the McKean-Vlasov forward anticipated backward stochastic differential equation (MV-FABSDE) system. In addition, we establish a necessary and sufficient stochastic  maximum principle for this class of mean field control problems with delay based on the differential calculus on function of measures, as well as existence and uniqueness results  for the associated MV-FABSDE system.

\end{abstract}




\onehalfspacing



\section{Introduction}\label{sec: intro}
Stochastic games were introduced to study the optimal behaviors of agents interacting with each other. They are used to study the topic of systemic risk in the context of finance. For example, in \citep{Carmona_Fouque_Sun:2015}, the authors proposed a linear quadratic inter-bank borrowing and lending model, and solved explicitly for the Nash equilibrium with a finite number of players. Later, this model was extended in \citep{Carmona_Fouque_Mousavi_Sun:2018} by considering  delay in the control in the state dynamic to account for the debt repayment. The authors analyzed the problem via a  probabilistic approach which relies on stochastic maximum principle, as well as via an analytic approach which is built on top of an infinite dimensional dynamic programming principle. 

Both mean field control and mean field games are used to characterize the asymptotic  behavior of a stochastic game as the number of players grows to infinity under the assumption that all the agents behave similarly, but with different notion of equilibrium. Mean field games consist of solving a  standard control problem, where the flow of measures is fixed,  and solving a fixed point problem such that this flow of measures matches the distribution of the dynamic of a representative agent. Whereas, a mean field control problem is  a nonstandard control problem in the sense that the law of state is present in the McKean-Vlasov dynamic, and optimization is performed while imposing the constraint of distribution of the state. More details can be found in \citep{Carmona_Delarue_vol:2018} and \citep{Bensoussan_book:2013}. 

In this paper, we consider a general class of mean field control problems with delay effect in the McKean-Vlasov dynamic. We derive the adjoint process associated with the delayed McKean-Vlasov  stochastic  differential equation, which is an anticipated backward stochastic differential equation of McKean-Vlasov type due to the fact that the conditional expectation of the future of adjoint process as well as the distribution of the state dynamic are involved. This type of anticipated backward stochastic differential equations (BSDE) was introduced in \citep{Peng_Yang:2009}, and for the general theory of BSDE, we refer to \citep{Zhang_book:2017}.  The necessary and sufficient part of stochastic maximum principle for control problem with delay in state and control can be found in \citep{Chen_Wu:2010}. Here, we also establish a necessary and sufficient stochastic maximum principle based on differential calculus on functions of measures as we consider the delay in the distribution.  In the meantime, we also prove the existence and uniqueness of the system of McKean-Vlasov  forward anticipated backward stochastic differential equations (MV-FABSDE) under some suitable conditions using the method of continuation, for which we also refer to \citep{Zhang_book:2017}, \citep{Peng_Wu:1999}, \citep{Bensoussan_Yam_Zhang:2015} and \citep{Carmona_Fouque_Mousavi_Sun:2018}. For a comprehensive study of FBSDE theory, we refer to \citep{Ma_Yong:2007}.

When there was no delay effect in the dynamic,  the relation between the solution to the FBSDE and quasi-linear partial differential equation (PDE) via "Four Step Scheme" is proved in \citep{Ma_Protter_Yong:1994}. The use of deep learning for solving these PDEs in high dimensions is explored in \citep{E_Han_Jentzen:2017} and \citep{Raissi:2018}. The class of fully coupled MV-FABSDE considered in our paper has no explicit solution, and we present an algorithm to tackle the above problem by means of deep learning techniques. Due to the non-Markovian nature of the state dynamic, we apply the long short-term memory (LSTM) network, which is able to capture the arbitrary long-term dependencies in the data sequence. It also partially solves the vanishing gradient problem in vanilla recurrent neural networks (RNNs), as was shown in \citep{Hochreiter_Schmidhuber:1997}. The idea of our algorithm is to approximate the solution to the adjoint process and the conditional expectation of the adjoint process. The optimal control is readily obtained after the MV-FABSDE being solved. We also emphasize that our numerical method for computing conditional expectations may have a wide range of applications, and it is simple to implement. We also present another algorithm solving the mean field control problem by directly parameterizing the optimal control. Similar idea can be found in  the policy gradient method in the regime of reinforcement learning \citep{Sutton:1999} as well as in \citep{E_Han:2016}. Numerically, the two algorithms that we propose in this paper yield the same results. Besides, our approaches are benchmarked to the case with no delay for which we have explicit solutions..

The paper is organized as follows. We start with an $N$-player game with delay, and let the number of players goes to infinity to introduce a mean field control problem in Section \ref{formulation}. Next, in Section \ref{learning}, we mathematically formulate the feedforward neural networks and LSTM networks, and we propose two algorithms to numerically solve the mean field control problem with delay using deep learning techniques. This is illustrated on a simple linear-quadratic toy model, however with delay in  the control. One algorithm is based on directly parameterizing the control, and the other depends on numerically solving the MV-FABSDE system. In addition, we also provide an example of  solving a linear quadratic mean field control problem with no delay both analytically, and numerically.  The adjoint process associated with the delayed dynamic is derived, as well as the stochastic maximum principle is proved in Section \ref{smp}. Finally, the uniqueness and existence solution for this class of MV-FABSDE are proved under suitable assumptions via continuation method in Section \ref{uniqueness}.

\section{Formulation of the Problem}\label{formulation}
 We consider an $N$-player game with delay in both state and control. The dynamic $(X_t^i)_{0 \leq t \leq T}$ for player $i \in \{1, \dots, N\}$ is given by a stochastic delayed differential equation (SDDE),
\begin{equation}\label{n-game}
\begin{aligned}
dX_t^i & = b^i(t, \X_t, \X_{t-\tau}, \alpha^i_t, \alpha^i_{t-\tau}) dt + \sigma^i(t, \X_t, \X_{t-\tau}, \alpha^i_t, \alpha^i_{t-\tau}) dW_t^i, \quad t \in (0, T],\\
X_0^i & = x^i_0, \\ 
X_t^i & = \alpha_t^i  = 0; \quad t \in [-\tau, 0),\\
\end{aligned}
\end{equation}
for $T > 0, \ \tau > 0$ given constants, where $\X_t = (X_t^1, \cdots, X_t^N)$, and where
$((W_t^i)_{t \in [0,T]})_{i = 1, \cdots, N}$ are $N$ independent Brownian motions defined on the space $(\W, \F, \P)$, $(\F_t)_{0 \leq t \leq T}$ being the natural filtration of Brownian motions.  
$$(b^i, \sigma^i): [0, T] \times \times \R^N \times \R^N \times  A \times A \to \R \times \R,$$
are  progressively measurable functions with values in $\R$. We denote $A$ a closed convex subset of $\R$, the set of actions that player $i$ can take, and denote $\A$ the set of admissible control processes. For each $i \in \{1, \dots, N\}$, $A$-valued measurable processes $(\alpha^i_t)_{0 \leq t \leq T}$ satisfy an integrability condition such that $\E\left[ \int_{-\tau}^T |\alpha_t^i|^2 dt \right] < + \infty$. 

Given an initial condition $\x_0 = (x_0^1, \cdots, x_0^N) \in \R^N$, each player would like to minimize his objective functional:
\begin{equation}\label{n-obj}
J^i(\a) = \E \left[\int_0^T f^i(t,  \X_t, \X_{t-\tau}, \alpha_t^i) dt + g^i(\X_T) \right],
\end{equation}
for some Borel measurable functions $f^i: [0, T]  \times \R^N \times \R^N \times A \to \R$, and  $g^i: \R^N \to \R$.

In order to study the mean-field limit of $(\X_t)_{t \in [0, T]}$, we assume that the system \eqref{n-game} satisfy a symmetric property, that is to say, for each player $i$, the other players are indistinguishable. Therefore, drift $b^i$ and volatility $\sigma^i$ in \eqref{n-game} take the form of 
$$(b^i, \sigma^i )(t, \X_t, \X_{t-\tau}, \alpha^i_t, \alpha^i_{t-\tau}) = (b^i,\sigma^i)(t, X_t^i, \mu_t^N, X_{t-\tau}^i, \mu_{t-\tau}^N, \alpha_t^i, \alpha_{t-\tau}^i), $$
and the running cost $f^i$ and terminal cost $g^i$ are of the form
$$f^i(t, \X_t, \X_{t-\tau}, \alpha_t^i) = f^i(t, X_t^i, \mu_t^N, X_{t-\tau}^i, \mu_{t-\tau}^N, \alpha_t^i) \mbox{ and }  g^i(\X_T) = g^i(X_T^i, \mu_T^N),$$
where we use the notation $\mu_t^N$ for the empirical distribution of $\X = (X^1, \cdots, X^N)$ at time $t$, which is defined as
$$\mu_t^N = \frac{1}{N} \sum_{j=1}^N \delta_{X_t^j}.$$

Next, we let the number of players $N$ go to $+\infty$ before we perform the optimization. According to symmetry property and the theory of propagation of chaos, the joint distribution of the $N$ dimensional process $(\X_t)_{0 \leq t \leq T} = (X_t^1, \dots, X_t^N)_{0 \leq t \leq T}$ converges to a product distribution, and the distribution of each single marginal process converges to the distribution of $(X_t)_{0 \leq t \leq T}$ of the following Mckean-Vlasov stochastic delayed differential equation (MV-SDDE). For more detail on the argument without delay, we refer to \citep{Carmona_Delarue_vol:2018} and \citep{Carmona_Delarue_Lachapelle:2014}.
\begin{equation}\label{MV_SDDE}
\begin{aligned}
dX_t & = b(t, X_t, \mu_t, X_{t-\tau}, \mu_{t-\tau}, \alpha_t, \alpha_{t-\tau}) dt + \sigma(t, X_t, \mu_t, X_{t-\tau}, \mu_{t-\tau}, \alpha_t, \alpha_{t-\tau}) dW_t, \quad t \in (0, T],\\
X_0 & = x_0, \\ 
X_t & = \alpha_t  = 0, \quad t \in [-\tau, 0).\\
\end{aligned}
\end{equation}
We then optimize after taking the limit. The objective for each player of  \eqref{n-obj} now becomes
\begin{equation}\label{MV_obj}
J(\alpha) = \E \left[ \int_0^T f(X_t, \mu_t, X_{t-\tau}, \mu_{t-\tau}, \alpha_t)dt + g(X_T, \mu_T)\right],
\end{equation}
where we denote $\mu_t := \L(X_t)$ the law of $X_t$.

\section{Solving Mean-Field Control Problems Using Deep Learning Techniques}\label{learning}

Due to the non-Markovian structure, the above mean-field optimal control problem \eqref{MV_SDDE}-\eqref{MV_obj} is difficult to solve either analytically or numerically. Here we propose two algorithms together with four approaches to tackle the above problem based on deep learning techniques. We would like to use two types of neural networks, one is called the feedforward neural network, and the other one is called Long Short-Term Memory (LSTM) network.

For a feedforward neural network, we first define the set of layers $\M_{d, h}^{\rho}$, for $x \in \R^d$,  as 
\begin{equation}
 \M_{d, h}^\rho :=  \{M: \R^d \to \R^h |  M(x) =  \rho(Ax + b), A \in \R^{h\times d}, b \in \R^h\}. 
\end{equation}
$d$ is called input dimension, $h$ is known as the number of hidden neurons, $A \in \R^{h \times d}$ is the weight matrix, $b \in \R^h$ is the bias vector, and $\rho$ is called the activation function. The following activation functions will be used in this paper, for some $x \in \R$,
$$\rho_{ReLU} (x) := x^+ = \max(0, x); \quad \rho_s (x) := \frac{1}{1 + e^{-x}}; \quad \rho_{\tanh} (x) := \tanh(x); \quad \rho_{Id}(x) := x.$$
Then feedforward neural network is defined as as a composition of layers, so that the set of feedforward neural networks with $l$ hidden layers we use in this paper is defined as 
\begin{multline}\label{NN}
\N\N_{d_1, d_2}^l = \{\tilde{M}: \R^{d_1} \to \R^{d_2}| \tilde{M} =  M_{l} \circ \cdots \circ M_1 \circ M_0, \\ 
M_0 \in  \M^{\rho_{ReLU}}_{d_{1}, h_1},  M_{l} \in  \M^{\rho_{Id}}_{h_{l}, d_2}, M_i \in \M^{\rho_{ReLU}}_{h_{i}, h_{i+1}}, h_{\cdot} \in \Z^+, i = 1, \dots, l-1\}.
\end{multline}

The LSTM network is one of  recurrent neural networks(RNN) architectures, which are powerful for capturing long-range dependence of the data. It is proposed in \citep{Hochreiter_Schmidhuber:1997}, and it is designed to solve the shrinking gradient effects which basic RNN often suffers from. The LSTM network is a chain of cells. Each LSTM cell  is composed  of a cell state, which contains information, and three gates, which regulate the flow of information. Mathematically, the rule inside $t$th cell follows,
\begin{equation}\label{gates}
\begin{aligned}
\Gamma_{f_t} = & \rho_{s} (A_f x_t + U_f a_{t-1} + b_f), \\
\Gamma_{i_t} = & \rho_{s}(A_i x_t + U_i a_{t-1} + b_i), \\
\Gamma_{o_t} = & \rho_s (A_o x_t + U_o a_{t-1} + b_o), \\
c_t = & \Gamma_{f_t} \odot c_{t-1} + \Gamma_{i_t} \odot  \rho_{\tanh} (A_c x_t + U_c a_{t-1} + b_c), \\
a_t = & \Gamma_{o_t} \odot  \rho_{\tanh}(c_t), \\
\end{aligned}
\end{equation}
where the operator $\odot $ denotes the  Hadamard product. 
$(\Gamma_{f_t}, \Gamma_{i_t}, \Gamma_{o_t}) \in \R^h \times \R^h \times \R^h$ represents forget gate, input gate and output gate respectively, $h$ refers the number of hidden neurons. $x_t \in \R^d$ is the input vector with $d$ features. $a_t \in \R^h$ is known as the output vector with initial value $a_0 = 0$, and $c_t \in \R^h$ is known as the cell state with initial value $c_0 = 0$. $A_\cdot \in \R^{h \times d}$ are the weight matrices connecting input and hidden layers,  $U_\cdot \in \R^{h \times h}$ are the weight matrix connecting hidden and output layers, and $b \in \R^h$  represents bias vector. The weight matrices and bias vectors are shared through all time steps, and are going to be learned during training process by back-propagation through time (BPTT), which can be implemented in Tensorflow platform. 
Here we define the set of LSTM network up to time $t$ as
\begin{multline}\label{LSTM}
\LSTM_{d, h, t} = \bigg\{M: (\R^d)^t \times \R^{h} \times \R^h \to \R^h \times \R^h \  | \ M(x_{0}, \dots, x_t, a_{0}, c_{0}) = (a_t, c_t), \\
c_t = \Gamma_{f_t} \odot c_{t-1} + \Gamma_{i_t} \odot  \rho_{\tanh} (A_c x_t + U_c a_{t-1} + b_c), a_t = \Gamma_{o_t} \odot  \rho_{\tanh}(c_t), a_0 = c_0 = 0  \bigg\},
\end{multline}
where $\Gamma_{f_\cdot}, \Gamma_{i_\cdot}, \Gamma_{o_\cdot}$ are defined in \eqref{gates}.

In particular, we specify the model in a linear-quadratic form, which is inspired by \citep{Carmona_Fouque_Mousavi_Sun:2018} and \citep{Fouque_Zhang:2018}. The objective function is defined as
\begin{equation}\label{SDDE_obj}
J(\alpha) = \E \left[\int_0^T \left(\frac{1}{2}\alpha_t^2 + \frac{c_f}{2} (X_t - m_t)^2 \right) dt + \frac{c_t}{2} (X_T - m_T)^2  \right],
\end{equation}
subject to
\begin{equation}\label{SDDE}
\begin{aligned}
dX_t = & (\alpha_t - \alpha_{t-\tau}) dt + \sigma dW_t, \quad t \in [0,T],\\
X_0 = & x_0,\\
X_t = \alpha_t = & 0, \quad t \in [-\tau, 0),
\end{aligned}
\end{equation}
where $\sigma, c_f, c_t > 0$ are given  constants, and $m_t := \int_\R x d\mu_t(x)$ denotes the mean of $X$ at time $t$, and $\mu_t := \L(X_t)$ .  
In the following subsections, we solve the above problem numerically using two algorithms together with four approaches. The first two approaches are to directly approximate the control by either a LSTM network or a feedforward neural network, and minimize the objective \eqref{SDDE_obj} using stochastic gradient descent algorithm.
The third and fourth approaches are to introduce the adjoint process associated with \eqref{SDDE}, and approximate the adjoint process and the  conditional expectation of adjoint process using neural networks.

\subsection{Approximating the Optimal Control Using Neural Networks}

We first set $\Delta t = T/ N = \tau / D$ for some positive integer $N$. The time discretization becomes $$-\tau = t_{-D} \leq t_{-D+1} \leq  \cdots \leq t_{0} = 0 = t_0 \leq t_1 \leq \cdots \leq t_N \leq T,$$ for $t_i - t_{i - 1} = \Delta t, \  i \in \{-D+1, \cdots, 0,  \cdots, N-1, N\}$. The  discretized  SDDE  \eqref{SDDE} according to Euler-Maruyama scheme now reads
\begin{equation}
X_{t_{i+1}} = X_{t_{i}} + (\alpha_{t_i} - \alpha_{t_{i-D}}) \Delta t + \sigma \sqrt{\Delta t} \Delta W_{t_{i}}, \mbox{ for } i \in \{0, \cdots, N-1\},
\end{equation}
where $(\Delta W_{t_i})_{0 \leq i \leq N-1}$ are independent, normal distributed sequence of random variables with mean 0 and variance 1. 

 First, from the definition of open loop control, and due to non-Markovian feature of \eqref{SDDE}, the open-loop optimal control is a function of the path of the Brownian motions up to time $t$, i.e., $\alpha(t,(W_s)_{0 \leq s \leq t})$. We are able to describe this dependency by a LSTM network by parametrizing the control as a function of current time and the discretized increments of Brownian motion path, i.e., 
\begin{equation}\label{alphaLSTM}
\begin{aligned}
(a_{t_i}, c_{t_i}) = & \varphi^1(t_i, (\Delta W_s)_{t_0 \leq s \leq t_i} | \Phi^1)\\
&  \mbox{ for $\varphi^1 \in \LSTM_{2, h_1, t}$ and $\Phi^1 = (A_f, A_i, A_o, A_c, U_f, U_i, U_o, U_c, b_f, b_i, b_o, b_c)$}, \\
\alpha(t_i, (W_s)_{t_0 \leq s \leq t_i}) \approx & \psi^1(a_{t_i} | \Psi^1 ) \mbox{ for $\psi^1 \in \M_{h_1, 1}^{Id}$ and $\Psi^1 = (A, b)$}, 
\end{aligned}
\end{equation} 
for some $h_1 \in \Z^+$. We remark that the last dense layer is used to match the desired output dimension.

The second approach is again directly approximate the control but with a feedforward neural network. Due to the special structure of our model, where the mean of dynamic in \eqref{SDDE} is constant, the mean field control problem coincides with the mean field game problem.  In \citep{Fouque_Zhang:2018}, authors solved the associated mean field game problem using infinite dimensional PDE approach, and found that the optimal control is a function of current state and the past of control. 
Therefore, the feedforward neural network with $l$ layers, which we use to approximate the optimal control, is defined as
\begin{equation}\label{alphaNN}
\begin{aligned}
\alpha_{t_i}(X_{t_i}, (\alpha_s)_{t_{i-D} \leq s < t_i}) \approx &  \psi^2(X_{t_i}, (\alpha_s)_{t_{i-D} \leq s \leq t_i} | \Psi^2)\\
& \mbox{ for } \psi^2 \in \NN_{D+1, 1}^l, \ \Psi^2 = (A_{0}, b_0, \dots, A_{l}, b_{l}).
\end{aligned}
\end{equation}

From Monte Carlo algorithm, and trapezoidal rule, the objective function \eqref{SDDE_obj} now becomes
\begin{multline}\label{obj_discrete}
J = \frac{1}{M} \sum_{j=1}^M \bigg[ \bigg( \frac{1}{2} (\alpha_{t_0}^{(j)} )^2 + \frac{c_f}{2}(X_{t_0}^{(j)} - \bar{X}_{t_0})^2 + \sum_{i=1}^{N-1} \bigg( ( \alpha_{t_i}^{(j)}  )^2 + c_f(X_{t_i}^{(j)} - \bar{X}_{t_i})^2 \bigg)  \\
+ \frac{1}{2} (\alpha_{t_N}^{(j)} )^2 + \frac{c_f}{2}(X_{t_N}^{(j)} - \bar{X}_{t_N})^2 \bigg) \frac{\Delta t}{2} + \frac{c_t}{2} (X_{t_N}^{(j)} - \bar{X}_{t_N})^2\bigg],
\end{multline}
where $M$ denotes the number of realizations and $\bar{X} := \frac{1}{M} \sum_{j=1}^M X^{(j)}$ denotes the sample mean. After plugging in the neural network either given by \eqref{alphaLSTM}  or \eqref{alphaNN}, the optimization problem becomes to find the best set of parameters either $(\Phi^1, \Psi^1)$ or $\Psi^2$ such that the objective $J(\Phi^1, \Psi^1)$ or $J(\Psi^2)$ is minimized with respect to those parameters. 

The algorithm works as follows:
\\
\begin{algorithm}[H]
\SetAlgoLined
 Initialization of parameters $\Theta_1 = (\Phi^1, \Psi^1)$ for approach 1 \eqref{alphaLSTM} or $\Theta_1 = (\Psi^2)$ for approach 2 \eqref{alphaNN}\;
 
 \For{each epoch $e = 1, 2, \dots$}{
  	$\bullet$ Generate $\Delta W \in \R^{M \times N}$ for $\Delta W_{t_i}^{(j)} := W_{ji} \sim N(0,1)$, $j \in \{1, \dots, M\}$ and $i \in \{1, \dots, N\}$ \;
  	$\bullet$ $ \alpha_{t_i}^{(j)} = 0 $ for $i = \{-D, \dots, -1\}$, $\forall j$\;
  	$\bullet$  $X_0^{(j)} = \bar{X}_0 = x_0$,  $\alpha_0^{(j)} \approx \varphi_0^{(j)}(\Theta_e), \forall j, $ for some network $\varphi$ given by \eqref{alphaLSTM} or by \eqref{alphaNN} at $t_0$ with proper inputs\;
  	$\bullet$  $J = \frac{1}{M}  \sum_{j=1}^M \frac{1}{2} (\varphi_0^{(j)})^2 \frac{\Delta t}{2}$ \;
  \For{($i = 0, \dots, N-1$)}{
   $\bullet$  $X_{t_{i+1}}^{(j)} = X_{t_{i}}^{(j)} + (\alpha_{t_i}^{(j)} - \alpha_{t_{i-D}}^{(j)}) \Delta t + \sigma \sqrt{\Delta t } \Delta W_{t_{i}}^{(j)}, \forall j$\;
   $\bullet$  $\bar{X}_{t_{i+1}} = \frac{1}{M}\sum_{j=1}^M X_{t_{i+1}}^{(j)}$ \;
   $\bullet$  $\alpha_{t_{i+1}}^{(j)} \approx \varphi_{t_{i+1}}^{(j)}(\Theta_e), \forall j$, is  given by either \eqref{alphaLSTM} or \eqref{alphaNN} at $t_{i+1}$ \;
  \eIf{($i =  N-1$)}{
   $\bullet$  $J = J+ \frac{1}{M}  \sum_{j=1}^M \bigg( \frac{1}{2}(\varphi_{t_{i+1}}^{(j)})^2 + \frac{c_f}{2}(X_{t_{i+1}}^{(j)} - \bar{X}_{t_{i+1}})^2 \bigg) \frac{\Delta t}{2}$\;
   }{
   $\bullet$  $J = J+ \frac{1}{M}  \sum_{j=1}^M \bigg( \frac{1}{2}(\varphi_{t_{i+1}}^{(j)})^2 + \frac{c_f}{2}(X_{t_{i+1}}^{(j)} - \bar{X}_{t_{i+1}})^2 \bigg) \Delta t$\;
  }
   }{
   $\bullet$  $J = J + \frac{1}{M}  \sum_{j=1}^M \frac{c_t}{2} (X_{t_N}^{(j)} - \bar{X}_{t_N})^2, \forall j$ \;
   $\bullet$   Compute the gradient $\nabla J(\Theta_e)$ by backpropagation through time \;
   $\bullet$  Stop if $J(\Theta_e)$ converges, or $|\nabla J(\Theta_e)| < \delta$ for some threshold $\delta$, and return $\Theta_e$\;
      $\bullet$  Otherwise, 
update $\Theta_{e+1} = \Theta_e - \eta \nabla J(\Theta_e),$ according to stochastic gradient descent algorithm, for some learning rate $\eta > 0$ small \;
  }
 }
 \caption{Algorithms for solving mean field control problem with delay by directly approximating the optimal control using neural networks}
\end{algorithm}

In the following graphics, we choose $x_0 = 0, c_f = c_t = 1, \sigma = 1, T = 10, \tau = 4, \Delta t = 0.1, M = 4000$. For approach 1, the neural network $\varphi \in \NN$, which is defined in \eqref{NN}, is composed of 3 hidden layers with $d_1 = 42, h_1 = 64, h_2 = 128, h_3 = 64, d_2 = 1$. For approach 2, the LSTM network $\varphi \in \LSTM$, which is defined in \eqref{LSTM}, consists of 128 hidden neurons.    For a specific representative path, the underlying Brownian motion paths are approximately the same for different approaches. Figure \ref{fig:x12} compares one representative optimal trajectory of the state dynamic and the control, and they coincide. Figure \ref{fig:mean12} plots the sample average of optimal trajectory of the state dynamic and the control,  which are trajectories of approximately 0, and this is the same as the theoretical mean.

\begin{figure}[h!]
\centering
  \includegraphics[scale=0.53]{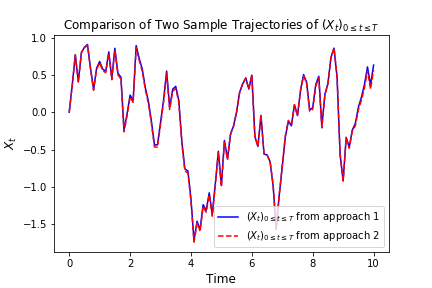}
  \includegraphics[scale=0.53]{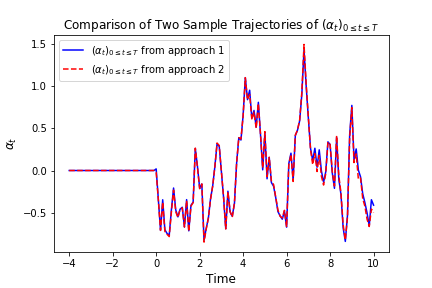}
  \caption{On the left, we compare one representative optimal trajectory of $(X_{t_i})_{t_0 \leq t_i \leq t_N}$. The plot on the right show the comparison of one representative optimal trajectory of $(\alpha_{t_i})_{t_0 \leq t_i \leq t_N}$ between approach 1 and approach 2.}
  \label{fig:x12}
\end{figure}

\begin{center}
\begin{figure}[h!]
  \includegraphics[scale=0.53]{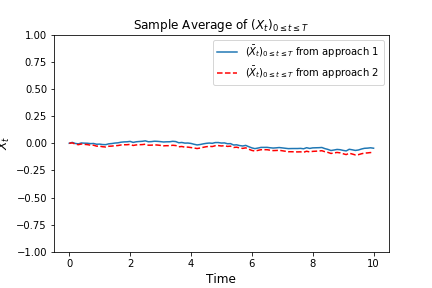}
\includegraphics[scale=0.53]{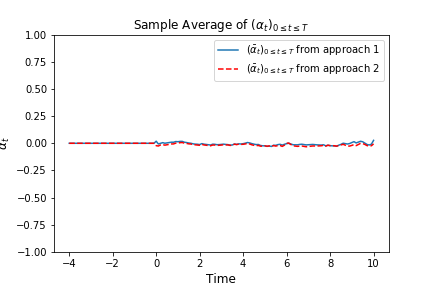}
  \caption{On the left, we compare the sample mean of  optimal trajectories of $(X_{t_i})_{t_0 \leq t_i \leq t_N}$. The plot on the right show the comparison of sample mean of  trajectories of optimal control $(\alpha_{t_i})_{t_0 \leq t_i \leq t_N}$ between approach 1 and approach 2.}
  \label{fig:mean12}
\end{figure}
\end{center}

\subsection{Approximating the Adjoint Process Using Neural Networks}

The third and fourth approaches are based on numerically solving the MV-FABSDE system using LSTM network and feedforward neural networks. From Section \ref{smp}, we derive the adjoint process, and prove the sufficient and necessary parts of stochastic maximum principle. From \eqref{adjoint},  we are able to write the backward stochastic differential equation associated to \eqref{SDDE} as,
\begin{equation}\label{BSDE}
dY_t = - c_f (X_t - m_t) dt + Z_t dW_t, \ t \in [0,T],
\end{equation}
with terminal condition $Y_T = c_t (X_T - m_T)$, and $Y_s = 0$ for $s \in (T, T+\tau]$. The optimal control $(\hat{\alpha}_t)_{0 \leq t \leq T}$ can be obtained in terms of the adjoint process $Y_t$ from the maximum principle, and it is given by 
$$\hat{\alpha}_t = - Y_t + \E[Y_{t+\tau} | \F_t].$$ 
From the Euler-Maruyama scheme, the discretized version of \eqref{SDDE} and \eqref{BSDE} now reads, for $i = \{0, \dots, N-1\}$,
\begin{align}\label{Euler}
X_{t_{i+1}}  = & X_{t_i} + (\hat{\alpha}_{t_i} - \hat{\alpha}_{t_{i-D}}) \Delta t + \sigma \sqrt{\Delta t} \Delta W_{t_{i}}, \mbox{where } \Delta W_{t_{i}} \sim N(0, 1). \\
Y_{t_{i+1}} = & Y_{t_i} - c_f (X_{t_i} - \bar{X}_{t_i})  \Delta t+ Z_{t_{i}}  \sqrt{\Delta t} \Delta W_{t_{i}},
\end{align}
where we use the sample average $\bar{X}_{t} = \frac{1}{M} \sum_{j=1}^M X_{t}^{(j)}$ to approximate the expectation of $X_{t}$. In order to solve the above MV-FABSDE system, we need to approximate $(Y_{t_i}, \E[Y_{t_{i+D}} | \F_{t_i}], Z_{t_i})_{0 \leq t_i \leq t_N}$. 

The third approach consists of approximating  $(Y_{t_i}, \E[Y_{t_{i+D}} | \F_{t_i}], Z_{t_i})_{0 \leq t_i \leq t_N}$ using three LSTM networks as functions of current time and the discretized path of Brownian motions respectively, i.e., 
\begin{equation}\label{YLSTM}
\begin{aligned}
(a_{t_i}^Y, c_{t_i}^Y) = & \varphi^Y(t_i, (\Delta W_s)_{t_0 \leq s \leq t_i} | \Phi^Y)\\
&  \mbox{ for $\varphi^Y \in \LSTM_{2, h_Y, t_i}$ and $\Phi^Y = (A^Y_f, A^Y_i, A^Y_o, A^Y_c, U^Y_f, U^Y_i, U^Y_o, U^Y_c, b^Y_f, b^Y_i, b^Y_o, b^Y_c)$}, \\
Y_{t_i} \approx & \psi^Y(a_{t_i}^Y | \Psi^Y ) \mbox{ for $\psi^Y \in \M_{h_Y, 1}^{Id}$ and $\Psi^Y = (A^Y, b^Y)$}, \\
\mbox{}\\
(a_{t_i}^{EY}, c_{t_i}^{EY}) = & \varphi^{EY}(t_i, (\Delta W_s)_{t_0 \leq s \leq t_i} | \Phi^{EY})\\
&  \mbox{ for $\varphi^{EY} \in \LSTM_{2, h_{EY}, t_i}$}\\
& \mbox{ and $\Phi^{EY} = (A^{EY}_f, A^{EY}_i, A^{EY}_o, A^{EY}_c, U^{EY}_f, U^{EY}_i, U^{EY}_o, U^{EY}_c, b^{EY}_f, b^{EY}_i, b^{EY}_o, b^{EY}_c)$}, \\
E[Y_{t_{i+D}} | \F_{t_i}] \approx & \psi^{EY}(a_{t_i}^{EY} | \Psi^{EY}) \mbox{ for $\psi^{EY} \in \M_{h_{EY}, 1}^{Id}$ and $\Psi^{EY} = (A^{EY}, b^{EY})$}, \\
\mbox{}\\
(a_{t_i}^Z, c_{t_i}^Z) = & \varphi^Z(t_i, (\Delta W_s)_{t_0 \leq s \leq t_i} | \Phi^Z)\\
&  \mbox{ for $\varphi^Z \in \LSTM_{2, h_Z, t_i}$ and $\Phi^Z = (A^Z_f, A^Z_i, A^Z_o, A^Z_c, U^Z_f, U^Z_i, U^Z_o, U^Z_c, b^Z_f, b^Z_i, b^Z_o, b^Z_c)$}, \\
Z_{t_i} \approx & \psi^Z(a_{t_i}^Z | \Psi^Z ) \mbox{ for $\psi^Z \in \M_{h_Z, 1}^{Id}$ and $\Psi^Z = (A^Z, b^Z)$}, \\
\end{aligned}
\end{equation} 
for some $h_Y, h_{EY}, h_Z \in \Z^+$. Again, the last dense layers are used to match the desired output dimension. 

Since approach 3 consists of three neural networks with large number of parameters, which is hard to train in general, we would like to make the following simplification in approach 4 for approximating $(Y_{t_i}, \E[Y_{t_{i+D}} | \F_{t_i}], Z_{t_i})_{t_0 \leq t_i \leq t_N}$ via combination of one LSTM network and three feedforward neural networks. Specifically,
\begin{equation}\label{YNN}
\begin{aligned}
(a_{t_i}, c_{t_i}) = & \varphi(t_i, (\Delta W_s)_{t_0 \leq s \leq t_i} | \Phi)\\
&  \mbox{ for $\varphi \in \LSTM_{2, h, t_i}$ and $\Phi = (A_f, A_i, A_o, A_c, U_f, U_i, U_o, U_c, b_f, b_i, b_o, b_c)$}, \\
Y_{t_i} \approx & \psi^Y(a_{t_i} | \Psi^Y ) \mbox{ for $\psi^Y \in \NN_{h, 1}^l$ and $\Psi^Y = (A^Y_0, b^Y_0, \dots, A_{l}^Y, b_{l}^Y)$}, \\
E[Y_{t_{i+D}} | \F_{t_i}] \approx & \psi^{EY}(a_{t_i} | \Psi^{EY} ) \mbox{ for $\psi^{EY} \in \NN_{h, 1}^l$ and $\Psi^{EY} = (A^{EY}_0, b^{EY}_0, \dots, A_{l}^{EY}, b_{l}^{EY})$}, \\
Z_{t_i} \approx & \psi^Z(a_{t_i} | \Psi^Z ) \mbox{ for $\psi^Z \in \NN_{h, 1}^l$ and $\Psi^Z = (A^Z_0, b^Z_0, \dots, A_{l}^Z, b_{l}^Z)$}. \\
\end{aligned}
\end{equation} 

In words, the algorithm works as follows. We first initialize the parameters $(\Theta^Y, \Theta^{EY}, \Theta^Z) = ((\Phi^Y, \Psi^Y), (\Phi^{EY}, \Psi^{EY}), (\Phi^Z, \Psi^Z))$  either in \eqref{YLSTM} or $(\Theta^Y, \Theta^{EY}, \Theta^Z) = ((\Phi, \Psi^Y),  \Psi^{EY}), \Psi^Z)$ in  \eqref{YNN}. At time 0, $X_0 = x_0$, $(Y_0, \E[Y_{t_D} | \F_0], Z_0) \approx (\varphi_0^Y(\Theta^Y), \varphi_0^{EY}(\Theta^{EY}), \varphi_0^Z(\Theta^Z))$for some network $(\varphi^Y, \varphi^{EY}, \varphi^Z)$ given by either \eqref{YLSTM} or \eqref{YNN}, and $\alpha_0 = -Y_{t_0} + \E[Y_{t_D} | {\cal F}_{t_0}]$. Next, we update $X_{t_{i+1}}$ and $Y_{t_{i+1}}$ according to \eqref{Euler}, and the solution to the backward equation at $t_{i+1}$ is denoted by $\tilde{Y}_{t_{i+1}}$.
In the meantime, $Y_{t_{i+1}}$ is also approximated by a neural network. In such case, we refer to $\tilde{Y}_{\cdot}$ as the label, and $Y_{\cdot}$ given by the neural network as the prediction. We would like to minimize the mean square error between these two. At time $T$,  $Y_{t_N}$ is also supposed to match $c_t(X_{t_N} - \bar{X}_{t_N})$, from the terminal condition of \eqref{BSDE}. In addition, the conditional expectation $\E[Y_{t_{i+D}} | \F_{t_i}]$ given by a neural network should be the best predictor of $\tilde{Y}_{t_{i+D}}$, which implies that we would like to find the set of parameters $\Theta^{EY}$ such that $\E[(\tilde{Y}_{t_{i+D}} - \varphi_{t_i}^{EY}(\Theta^{EY}))^2 ]$ is minimized for all $t_i \in \{t_0, \dots, t_{N-D}\}$. Therefore, for $M$ samples, we would like to minimize two objective functions $L_1$ and $L_2$ defined as
\begin{equation}\label{l1l2}
\begin{aligned}
L_1 (\Theta^{Y}, \Theta^Z)= & \frac{1}{M} \bigg[ \sum_{j=1}^M \sum_{i=0}^N (\varphi^{Y, {(j)}}_{t_i}- \tilde{Y}_{t_{i}}^{(j)})^2 
+ \sum_{j=1}^M \left(\varphi^{Y, {(j)}}_{t_N} - c_t ( X_{t_N}^{(j)} - \bar{X}_{t_N}) \right)^2 \bigg],\\
L_2 (\Theta^{EY}) = & \frac{1}{M} \sum_{j=1}^M \sum_{i=0}^{N-D} 
\left(\varphi^{EY, (j)}_{t_{i}}  - \tilde{Y}_{t_{i + D}}^{(j)} \right)^2.
\end{aligned}
\end{equation}

The algorithm works as follows:\\
\begin{algorithm}[H]
\SetAlgoLined
 Initialization of parameters $(\Theta_1^Y, \Theta_1^{EY}, \Theta_1^Z)$ for approach 3 as in \eqref{YLSTM}
or  approach 4 as in \eqref{YNN}\;
 
 \For{each epoch $e = 1, 2, \dots$}{
  	$\bullet$ Generate $\Delta W \in \R^{M \times N}$ for $\Delta W_{t_i}^{(j)} := \Delta W_{ji} \sim N(0,1)$, $j \in \{1, \dots, M\}$ and $i \in \{1, \dots, N\}$ \;
  	$\bullet$   $\alpha_{t_i}^{(j)} = 0$ for $i \in \{-D, \dots, -1\}, \forall j$ \;
  	$\bullet$  $X_0^{(j)} = \bar{X}_0 = x_0,  \forall j$;  \quad $\left(Y_0^{(j)}, \E[Y_{t_D}^{(j)} | \F_0], Z_0^{(j)} \right) \approx \left(\varphi_0^{Y, (j)}(\Theta_e^Y), \varphi_0^{EY, (j)}(\Theta_e^{EY}), \varphi_0^{Z, (j)}(\Theta_e^Z) \right)$ given by either \eqref{YLSTM} or by \eqref{YNN} ;  \ $\alpha_0^{(j)} \approx -\varphi_0^{Y, (j)} + \varphi_0^{EY, (j)}  $ at $t_0$ \;
  	$\bullet$  $L_1 (\Theta^Y_e, \Theta^Z_e) = 0, \ L_2(\Theta^{EY}_e) = 0$ \;
  \For{($i = 0, \dots, N-1$)}{
   $\bullet$  $X_{t_{i+1}}^{(j)} = X_{t_{i}}^{(j)} + (\alpha_{t_i}^{(j)} - \alpha_{t_{i-D}}^{(j)}) \Delta t + \sigma \sqrt{\Delta t} \Delta W_{t_{i}}^{(j)}, \forall j$\;
   $\bullet$  $\bar{X}_{t_{i+1}} = \frac{1}{M}\sum_{j=1}^M X_{t_{i+1}}^{(j)}$ \;
   $\bullet$ $\tilde{Y}_{t_{i+1}}^{(j)} =  Y_{t_i}^{(j)} - c_f (X_{t_i}^{(j)} - \bar{X}_{t_i})  \Delta t + Z_{t_{i}} \sqrt{\Delta t} \Delta W_{t_{i}}^{(j)}, \forall j$ \;
  \eIf{($i \leq N - D$)}{
     $\bullet$   $\left(Y_{t_{i+1}}^{(j)}, \E[Y_{t_{i+1 + D}}^{(j)} | \F_{t_{i+1}}], Z_{t_{i+1}}^{(j)} \right) \approx \left(\varphi_{t_{i+1}}^{Y, (j)}(\Theta^Y_e), \varphi_{t_{i+1}}^{EY, (j)}(\Theta^{EY}_e), \varphi_{t_{i+1}}^{Z, (j)}(\Theta^Z_e) \right) , \forall j$ given by \eqref{YLSTM} or by \eqref{YNN} at $t_{i+1}$ \;
   	$\bullet$  $L_1 +=\frac{1}{M} \sum_{j=1}^M  \left( \varphi_{t_{i+1}}^{Y, (j)}- \tilde{Y}_{t_{i+1}}^{(j)} \right)^2 
$\;
   }{
   $\bullet$   $\left(Y_{t_{i+1}}^{(j)}, \E[Y_{t_{i+1 + D}}^{(j)} | \F_{t_{i+1}}], Z_{t_{i+1}}^{(j)} \right) \approx \left(\varphi_{t_{i+1}}^{Y, (j)}(\Theta_e^Y), 0, \varphi_{t_{i+1}}^{Z, (j)}(\Theta^Z_e) \right), \forall j$ given by \eqref{YLSTM} or by \eqref{YNN} at $t_{i+1}$ \;
   	$\bullet$  $L_1 +=\frac{1}{M} \sum_{j=1}^M  \left( \varphi_{t_{i+1}}^{Y, (j)} -   \tilde{Y}_{t_{i + 1}}^{(j)} \right)^2 
$\;
  }
   }
     \For{($i = 0, 1, \dots, N-D$)}{

   	$\bullet$  $L_2 +=\frac{1}{M} \sum_{j=1}^M  \left( \varphi_{t_i}^{EY, (j)} - \tilde{Y}_{t_{i+D}}^{(j)} \right)^2 
$\;

   }
   {
   $\bullet$  $L_1 += \frac{1}{M}\sum_{j=1}^M \left(\varphi^{Y, {(j)}}_{t_N}- c_t ( X_{t_N}^{(j)} - \bar{X}_{t_N} ) \right)^2$ \;
   $\bullet$  Compute the gradient $\nabla L_1(\Theta^Y, \Theta^Z)$ and $\nabla L_2(\Theta^{EY})$ by backpropagation through time\;
      $\bullet$  Stop if $L_1(\Theta^Y, \Theta^Z)$ are close to 0, and $L_2(\Theta^{EY})$ converges, return $(\Theta^Y, \Theta^{EY}, \Theta^Z)$\;
         $\bullet$  Otherwise, update $\Theta^Y_{e+1}, \Theta^Z_{e+1}$ and $\Theta^{EY}_{e+1}$ according to SGD algorithm\;
  }
 }
 \caption{Algorithms for solving mean field control problem with delay according to MV-FABSDE}
\end{algorithm}

Again, in the following graphics, we choose  $x_0 = 0, c_f = c_t = 1, \sigma = 1, T = 10, \tau = 4, \Delta t = 0.1, M = 4000$. In approach 3, each of the three LSTM networks approximating $Y_{t_i}, \E[Y_{t_i+D} | \F_{t_i}]$ and $Z_{t_i}$ consists of 128 hidden neurons respectively. In approach 4, the LSTM consists of 128 hidden neurons, and each of the feedforward neural networks has parameters $d_1 = 128, h_1 = 64, h_2 = 128, h_3 = 64, d_2 = 1$. For a specific representative path, the underlying Brownian motion paths are the same for different approaches. Figure \ref{fig:x34} compares one representative optimal trajectory of the state dynamic and the control via two approaches, and they coincide. Figure \ref{fig:mean34} plots the sample average of optimal trajectory of the dynamic and the control,  which are trajectories of 0, which is the same as the theoretical mean.  Comparing to Figure \ref{fig:x12} and Figure \ref{fig:mean12}, as well as based on numerous experiments, we find that given a path of Brownian motion, the two algorithms would yield similar optimal trajectory of state dynamic and similar path for the optimal control.  From Figure \ref{fig:l1l2}, the loss $L_1$ as defined in \eqref{l1l2} becomes approximately 0.02 in 1000 epochs for both approach 3 and approach 4. This can also be observed from Figure \ref{fig:yey}, since the red dash line and the blue solid line coincide for both left and right graphs. In addition, from the righthand side of Figure \ref{fig:l1l2}, we observe the loss $L_2$ as defined in  \eqref{l1l2} converges to 50 after 400 epochs. This is due to the fact that the conditional expectation can be understood as an orthogonal projection. Figure \ref{fig:z} plots 64 sample paths of the process $(Z_{t_i})_{t_0 \leq t_i \leq t_N}$, which seems to be a deterministic function since $\sigma$ is constant in this example. Finally, Figure \ref{fig:value} shows the convergence of the value function as number of epochs increases. Both algorithms arrive approximately at the same optimal value which is around 6 after 400 epochs. This confirms that the out control problem has a unique solution. In section \ref{uniqueness}, we show that the MV-FASBDE system is uniquely solvable.  It is also observable that the first algorithm converges faster than the second one, since it directly paramerizes the control using one neural network, instead of solving the MV-FABSDE system, which uses three neural networks.

\begin{center}
\begin{figure}[h!]
  \includegraphics[scale=0.53]{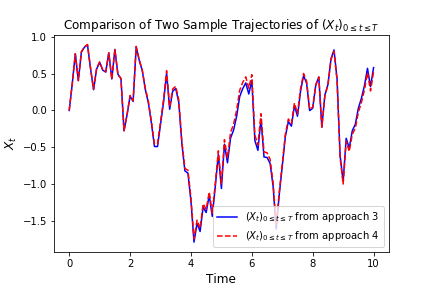}
\includegraphics[scale=0.53]{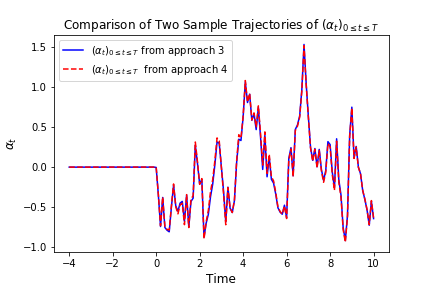}
  \caption{On the left, we compare one representative optimal trajectory of $(X_{t_i})_{t_0 \leq t_i \leq t_N}$. The plot on the right shows the comparison of one representative optimal trajectory of $(\alpha_{t_i})_{t_0 \leq t_i \leq t_N}$ between approach 3 and approach 4.}
  \label{fig:x34}
\end{figure}
\end{center}

\begin{center}
\begin{figure}[h!]
  \includegraphics[scale=0.53]{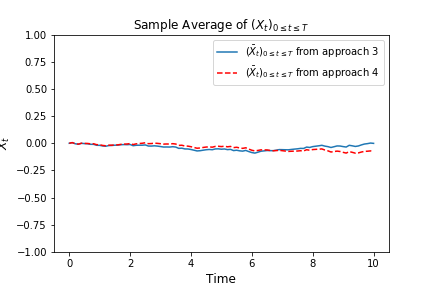}
\includegraphics[scale=0.53]{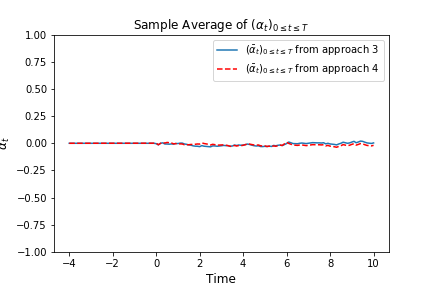}
  \caption{On the left, we compare the sample mean of  optimal trajectories of $(X_{t_i})_{t_0 \leq t_i \leq t_N}$. The plot on the right shows the comparison of sample mean of  trajectories of optimal control $(\alpha_{t_i})_{t_0 \leq t_i \leq t_N}$ between approach 3 and approach 4.}
  \label{fig:mean34}
\end{figure}
\end{center}

\begin{figure}[h!] 
\centering
  \includegraphics[scale=0.53]{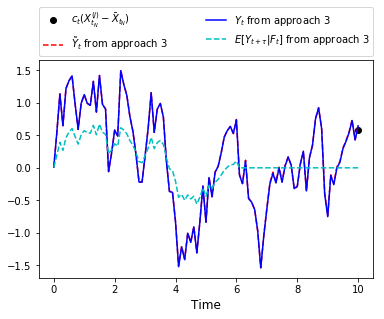}
  \quad
 \includegraphics[scale=0.53]{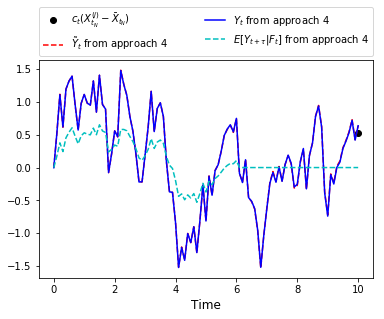}
  \caption{Plots of representative trajectories of ($(Y_{t_i})_{t_0 \leq t_i \leq t_N}, (\tilde{Y}_{t_i})_{t_0 \leq t_i \leq t_N},  (\E[Y_{t_{i+D}} | \F_{t_i}])_{t_0 \leq t_i \leq t_N} $), from approach 3 (one the left) and from approach 4 (on the right).}
  \label{fig:yey} 
\end{figure}
\begin{center}
\begin{figure}[h!]
  \includegraphics[scale=0.53]{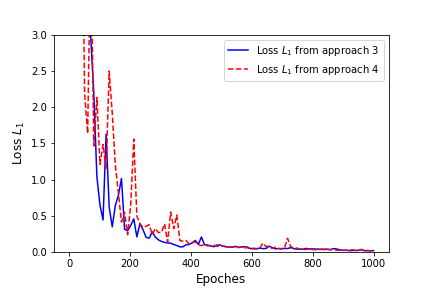}
\includegraphics[scale=0.53]{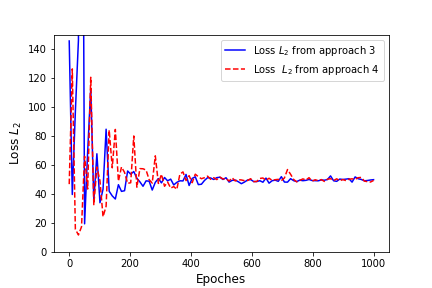}
  \caption{Convergence of loss $L_1$ (on the left) and convergence of loss $L_2$ (on the right) as defined in \eqref{l1l2} from approach 3 and approach 4.}
  \label{fig:l1l2} 
\end{figure}
\end{center}
\begin{figure}[h!]
\centering
  \includegraphics[scale=0.53]{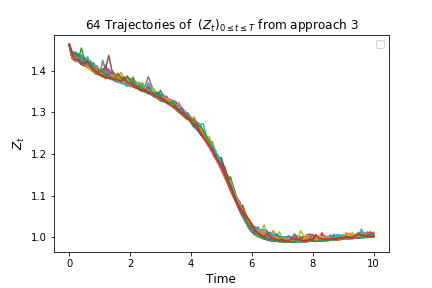}
\includegraphics[scale=0.53]{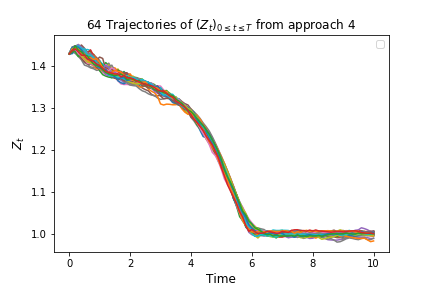}
  \caption{64 trajectories of $(Z_{t_i})_{t_0 \leq t_i \leq t_N}$ based on approach 3 (on the left) and approach 4 (on the right).}
  \label{fig:z} 
\end{figure}
\begin{figure}[h]
\centering
  \includegraphics[scale=0.8]{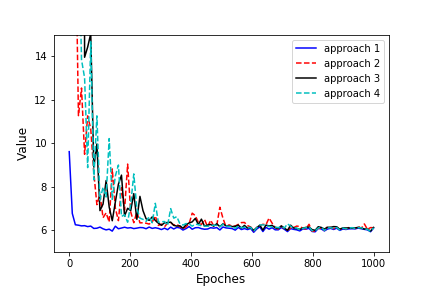}
  \caption{Comparison convergence of objective values as in $\eqref{obj_discrete}$ among four approaches.}
  \label{fig:value} 
\end{figure}

\subsection{Numerically Solving the Optimal Control Problem with No Delay}
Since the algorithms we proposed embrace the case with no delay, we illustrate the comparison between numerical results and the analytical results. By letting $\tau>T$ we obtain $\alpha_{t-\tau} = 0$ in \eqref{SDDE}, and we aim at solving the following linear-quadratic mean-field control problem by minimizing
\begin{equation}
J(\alpha) = \E \left[\int_0^T \left(\frac{1}{2}\alpha_t^2 + \frac{c_f}{2} (X_t - m_t)^2 \right) dt + \frac{c_t}{2} (X_T - m_T)^2  \right],
\end{equation}
subject to
\begin{equation}\label{forward_nodelay}
\begin{aligned}
dX_t = & \alpha_t  dt + \sigma dW_t, \quad t \in [0,T],\\
X_0 = & x_0.
\end{aligned}
\end{equation}
Again, from Section \ref{smp}, we find the optimal control
$$\hat{\alpha}_t = -Y_t,$$
where $(Y_t, Z_t)$ is the solution of the following adjoint process,
\begin{equation}\label{adjoint_nodelay}
dY_t = -c_f(X_t - m_t) dt + Z_t dW_t.
\end{equation}
Next, we make the ansatz
\begin{equation}\label{ansatz}
Y_t = \phi_t (X_t - m_t),
\end{equation}
for some deterministic function $\phi_t$, satisfying the terminal condition $\phi_T = c_t.$ Differentiating the ansatz, the backward equation should satisfy
\begin{equation}
dY_t = (\dot{\phi}_t - \phi_t^2) (X_t - m_t) dt + \phi_t \sigma dW_t,
\end{equation}
where $\dot{\phi}_t$ denotes the time derivative of $\phi_t$. Comparing with \eqref{adjoint_nodelay}, and identifying the drift and volatility term, $\phi_t$ must satisfy the scalar Riccati equation,
\begin{equation}\label{riccati}
\begin{cases}
\dot{\phi}_t = \phi_t^2 - c_f, \\
\phi_T = c_t,
\end{cases}
\end{equation}
and the process $Z_t$ should satisfy
\begin{equation}
Z_t = \phi_t \sigma,
\end{equation}
which is deterministic.  If we choose $x_0 = 0, \ c_f = c_t = 1, T = 10$, $\phi_t = 1$ solves the Riccati equation \eqref{riccati}, so that $Z_t = 1, \ \forall t \in [0,T]$, and from \eqref{ansatz}, the optimal control satisfies
\begin{equation}\label{alpha_optimal}
\hat{\alpha}_t = -(X_t - m_t).
\end{equation}

Numerically, we apply the two deep learning algorithms proposed in the previous section. The first algorithm directly approximates the control. According to the open loop formulation, we set 
\begin{equation}
\begin{aligned}
(a_{t_i}, c_{t_i}) = & \varphi(t_i, (\Delta W_s)_{t_0 \leq s \leq t_i} | \Phi)\\
&  \mbox{ for $\varphi \in \LSTM_{2, h, t}$ and $\Phi = (A_f, A_i, A_o, A_c, U_f, U_i, U_o, U_c, b_f, b_i, b_o, b_c)$}, \\
\alpha(t_i, (W_s)_{t_0 \leq s \leq t_i}) \approx & \psi(a_{t_i} | \Psi ) \mbox{ for $\psi \in \M_{h, 1}^{Id}$ and $\Psi= (A, b)$}, 
\end{aligned}
\end{equation} 
for some $h \in \Z^+$. We remark that the last dense layer is used to match the desired output dimension. The second algorithm numerically solves the forward backward system as in \eqref{forward_nodelay} and \eqref{adjoint_nodelay}. From the ansatz \eqref{ansatz} and the Markovian feature, we approximate $(Y_t, Z_t)_{0 \leq t \leq T}$ using two feedforward neural networks, i.e.,
\begin{equation}
\begin{aligned}
Y_{t_i} \approx &  \psi^1(t_i, X_{t_i} | \Psi^1) \mbox{ for } \psi^1 \in \NN_{2, 1}^l, \ \Psi^1 = (A_{0}^1, b_0^1, \dots, A_{l}^1, b_{l}^1); \\
Z_{t_i} \approx &  \psi^2(t_i, X_{t_i} | \Psi^2) \mbox{ for } \psi^2 \in \NN_{2, 1}^l, \ \Psi^2 = (A_{0}^2, b_0^2, \dots, A_{l}^2, b_{l}^2). \\
\end{aligned}
\end{equation}
Figure \ref{fig:x_alpha_no_delay} shows the representative optimal trajectory of $(X_{t_i} - \bar{X}_{t_i})_{t_0 \leq t_i \leq t_N}$ and $(\alpha_{t_i})_{t_0 \leq t_i \leq t_N}$ from both algorithms, which are exactly the same. The symmetry feature can be seen from the computation  \eqref{alpha_optimal}. On the left of Figure \ref{fig:x_alpha_mean_no_delay} confirms the mean of the processes $(X_{t_i} - \bar{X}_{t_i})_{t_0 \leq t_i \leq t_N}$ and $(\alpha_{t_i})_{t_0 \leq t_i \leq t_N}$  are 0 from both algorithms. The right picture of Figure \ref{fig:x_alpha_mean_no_delay} plots the data points of $x$ against $\alpha$, and we can observe that the optimal control $\alpha$ is linear in $x$ as a result of \eqref{alpha_optimal}, and the slope tends to be -1, since $\phi = 1$ solves the scalar Riccati equation \eqref{riccati}. Finally, Figure \ref{fig:yz_no_delay} plots representative optimal trajectories of the solution to the adjoint equations $(Y_t, Z_t)$. On the left, we observe that the adjoint process $(Y_t)_{0 \leq t \leq T}$ matches the terminal condtion, and on the right, $(Z_t)_{0 \leq t \leq T}$ appears to be a deterministic process of value $1$, and this matches the result we compute previously.

\begin{figure}[h!]
\centering
  \includegraphics[scale=0.53]{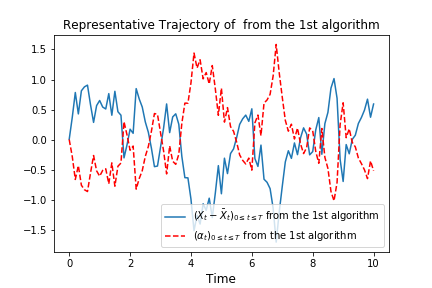}
\includegraphics[scale=0.53]{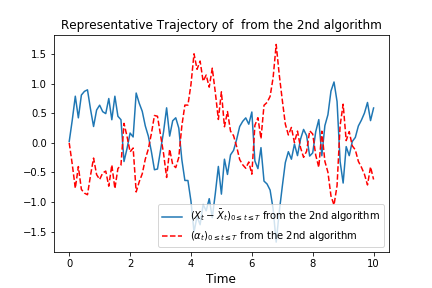}
  \caption{Representative optimal trajectory of$(X_{t_i} - \bar{X}_{t_i})_{t_0 \leq t_i \leq t_N}$ and $(\alpha_{t_i})_{t_0 \leq t_i \leq t_N}$ from algorithm 1 (on the left) and from algorithm 2 (on the right).}
  \label{fig:x_alpha_no_delay} 
\end{figure}
\begin{figure}[h!]
\centering
  \includegraphics[scale=0.53]{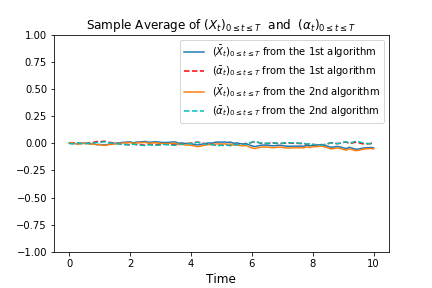}
\includegraphics[scale=0.53]{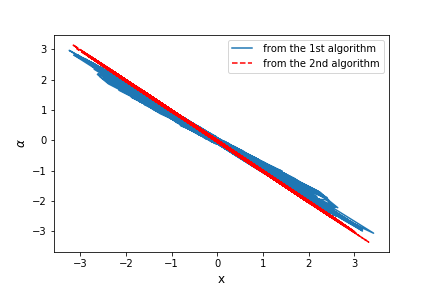}
  \caption{The picture on the left plots the sample averages of $(X_{t_i})_{t_0 \leq t_i \leq t_N}$ and $(\alpha_{t_i})_{t_0 \leq t_i \leq t_N}$ for both algorithm 1 and 2; The plot on the right shows the points of $x$ against $\alpha$ for both algorithms. }
  \label{fig:x_alpha_mean_no_delay} 
\end{figure}
\begin{figure}[h!]
\centering
  \includegraphics[scale=0.53]{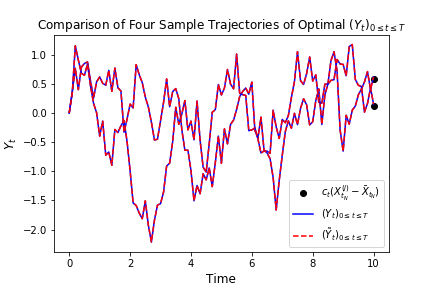}
\includegraphics[scale=0.53]{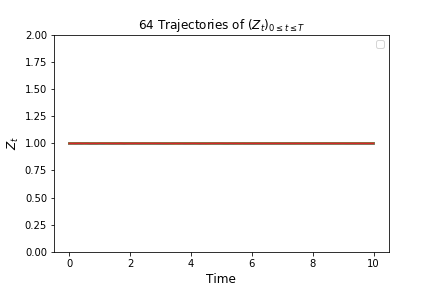}
	\caption{The plot on the left shows representative trajectories of ($(Y_{t_i})_{t_0 \leq t_i \leq t_N}, (\tilde{Y}_{t_i})_{t_0 \leq t_i \leq t_N})$. The picture on the right plots 64 trajectories of $(Z_{t_i})_{t_0 \leq t_i \leq t_N}$ .}
  \label{fig:yz_no_delay} 
\end{figure}

\section{Stochastic Maximum Principle for Optimality} \label{smp}

In this section, we derive the adjoint equation associated to our mean field stochastic control problem \eqref{MV_SDDE} and \eqref{MV_obj}. The necessary and sufficient parts of stochastic maximum principle have been proved for optimality. We assume 
\begin{enumerate}[label = (H4.\arabic*)]
\item $b, \sigma$ are differentiable with respect to $(X_t, \mu_t, X_{t-\tau}, \mu_{t-\tau}, \alpha_t, \alpha_{t-\tau})$; $f$ is differentiable with respect to $(X_t, \mu_t, X_{t-\tau}, \mu_{t-\tau}, \alpha)$; $g$ is differentiable with respect to $(X_T, \mu_T)$. Their derivatives are bounded.
\end{enumerate}

In order to simplify our notations, let $\theta_t = (X_t, \mu_t, \alpha_t)$.  For $0 < \e < 1$, we denote $\alpha^\e$ the admissible control defined by
$$\alpha_t^\e := \alpha_t + \epsilon (\beta_t - \alpha_t) := \alpha_t + \e \Delta \alpha_t,$$
for any $(\alpha)_{0 \leq t \leq T}$ and $(\beta)_{0 \leq t \leq T} \in \A.$ $X_t^{\e}:= X_t^{\alpha^{\e}}$ is the corresponding controlled process. We define $$\nabla X_t := \lim_{\e \to 0} \frac{X_t^\e - X_t^{\alpha}}{\e}$$ to be the variation process, which should follow the following dynamic for $t \in (0, T]$,
\begin{equation}
\begin{aligned}
d \nabla X_t  = & \bigg[ \partial_x b(t, \theta_t, \theta_{t-\tau})\nabla X_{t} + \partial_{x_{\tau}}b(t, \theta_t, \theta_{t-\tau}) \nabla X_{t-\tau} + \tilde{\E} [\partial_\mu b(t, \theta_t, \theta_{t-\tau}) (\tilde{X}_t) \nabla \tilde{X}_t] \\
	&  + \tilde{\E} [\partial_{\mu_\tau} b(t, \theta_t, \theta_{t-\tau}) (\tilde{X}_{t-\tau}) \nabla \tilde{X}_{t-\tau}] + \partial_{\alpha} b(t, \theta_t, \theta_{t-\tau})\Delta \alpha_t + \partial_{\alpha_\tau} b(t, \theta_t, \theta_{t-\tau})\Delta \alpha_{t-\tau} \bigg] dt\\
	& + \bigg[\partial_x \sigma(t, \theta_t, \theta_{t-\tau}) \nabla X_t + \partial_{x_{\tau}}\sigma(t, \theta_t, \theta_{t-\tau}) \nabla X_{t-\tau} + \tilde{\E} [\partial_\mu \sigma(t, \theta_t, \theta_{t-\tau}) (\tilde{X_t}) \nabla \tilde{X_t}] \\
	&  + \tilde{\E} [\partial_{\mu_\tau} \sigma(t, \theta_t, \theta_{t-\tau}) (\tilde{X}_{t-\tau}) \nabla \tilde{X}_{t-\tau}] + \partial_{\alpha} \sigma(t, \theta_t, \theta_{t-\tau})\Delta \alpha_t + \partial_{\alpha_\tau} \sigma(t, \theta_t, \theta_{t-\tau})\Delta \alpha_{t-\tau} \bigg] dW_t  \\
\end{aligned}
\end{equation}
with initial condition $\nabla X_t = \Delta \alpha_t = 0, \  t\in [-\tau, 0)$. $(\tilde{X}_t, \nabla \tilde{X}_t)$ is a copy of $(X_t, \nabla X_t)$ defined on $(\tilde{\W}, \tilde{\F}, \tilde{\P})$, where we apply differential calculus on functions of measure, see \citep{Carmona_Delarue_vol:2018} for detail. $\partial_x b,  \partial_{x_\tau} b, \partial_\mu b, \partial_{\mu_\tau} b,  \partial_\alpha b, \partial_{\alpha_\tau} b$ are derivatives of $b$ with respect to $(X_t, X_{t-\tau}, \mu_t, \mu_{t-\tau}, \alpha_t, \alpha_{t-\tau})$ respectively, and we use the same notation for $\partial_{\cdot} \sigma$.

In the meantime, the Gateaux derivative of functional $\alpha \to J(\alpha)$ is given by 
\begin{equation}\label{j-diff}
\begin{aligned}
& \lim_{\e \to 0} \frac{J(\alpha^\e) - J(\alpha)}{\e}\\
= & \E\left[ \partial_x g(X_T, \mu_T) \nabla X_T + \tilde{\E}[\partial_\mu g(X_T, \mu_T)(\tilde{X}_T) \nabla \tilde{X}_T]\right] \\
	& + \E \int_0^T \bigg[\partial_x f(\theta_t, X_{t-\tau}, \mu_{t-\tau}) \nabla X_t + \tilde{\E} [\partial_\mu f(\theta_t, X_{t-\tau}, \mu_{t-\tau})(\tilde{X}_t) \nabla \tilde{X}_t ] \\
	& + \partial_{x_{\tau}}f(\theta_t, X_{t-\tau}, \mu_{t-\tau}) \nabla X_{t-\tau} + \tilde{\E} [\partial_{\mu_\tau} f(\theta_t, X_{t-\tau}, \mu_{t-\tau}) (\tilde{X}_{t-\tau}) \nabla \tilde{X}_{t-\tau}] \\
	& + \partial_{\alpha} f(\theta_t, X_{t-\tau}, \mu_{t-\tau})(\Delta \alpha_t) \bigg] dt\\
\end{aligned}
\end{equation}

In order to determine the adjoint backward equation of $(Y_t, Z_t)_{0 \leq t \leq T}$ associated to \eqref{MV_SDDE}, we assume it is of the following form:
\begin{equation}
\begin{aligned}
dY_t & = -\varphi_t dt + Z_t dW_t, \quad t \in [0,T], \\
Y_T & = \partial_x g(X_T, \mu_T) + \tilde{\E}[\partial_\mu g(X_T, \mu_T) (\tilde{X}_T)],\\
Y_t & = Z_t = 0, \quad t \in (T, T+\tau]
\end{aligned}
\end{equation}
Next, we apply integration by part to $\nabla X_t$ and $Y_t$. It yields
\begin{align*}
& d(\nabla X_t Y_t)\\
= & Y_t \bigg[ \partial_x b(t, \theta_t, \theta_{t-\tau}) \nabla X_{t} + \partial_{x_{\tau}}b(t, \theta_t, \theta_{t-\tau}) \nabla X_{t-\tau} + \tilde{\E} [\partial_\mu b(t, \theta_t, \theta_{t-\tau}) (\tilde{X}_t) \nabla \tilde{X}_t] \\
	&  + \tilde{\E} [\partial_{\mu_\tau} b(t, \theta_t, \theta_{t-\tau}) (\tilde{X}_{t-\tau}) \nabla \tilde{X}_{t-\tau}] + \partial_{\alpha} b(t, \theta_t, \theta_{t-\tau})\Delta \alpha_t + \partial_{\alpha_\tau} b(t, \theta_t, \theta_{t-\tau})\Delta \alpha_{t-\tau} \bigg] dt\\
	& + Y_t \bigg[\partial_x \sigma(t, \theta_t, \theta_{t-\tau}) + \partial_{x_{\tau}}\sigma(t, \theta_t, \theta_{t-\tau}) \nabla X_{t-\tau} + \tilde{\E} [\partial_\mu \sigma(t, \theta_t, \theta_{t-\tau}) (\tilde{X_t}) \nabla \tilde{X_t} ]\\
	&  + \tilde{\E} [\partial_{\mu_\tau} \sigma(t, \theta_t, \theta_{t-\tau}) (\tilde{X}_{t-\tau}) \nabla \tilde{X}_{t-\tau}] + \partial_{\alpha} \sigma(t, \theta_t, \theta_{t-\tau})\Delta \alpha_t + \partial_{\alpha_\tau} \sigma(t, \theta_t, \theta_{t-\tau})\Delta \alpha_{t-\tau} \bigg] dW_t\\
	& - \varphi_t \nabla X_t dt + \nabla X_t Z_t dW_t\\
	& + Z_t \bigg[\partial_x \sigma(t, \theta_t, \theta_{t-\tau}) \nabla X_t + \partial_{x_{\tau}}\sigma(t, \theta_t, \theta_{t-\tau}) \nabla X_{t-\tau} + \tilde{\E} [\partial_\mu \sigma(t, \theta_t, \theta_{t-\tau}) (\tilde{X_t}) \nabla \tilde{X_t}] \\
	&  + \tilde{\E} [\partial_{\mu_\tau} \sigma(t, \theta_t, \theta_{t-\tau}) (\tilde{X}_{t-\tau}) \nabla \tilde{X}_{t-\tau}] + \partial_{\alpha} \sigma(t, \theta_t, \theta_{t-\tau})\Delta \alpha_t + \partial_{\alpha_\tau} \sigma(t, \theta_t, \theta_{t-\tau})\Delta \alpha_{t-\tau} \bigg] dt
\end{align*}
We integrate from 0 to $T$, and take expectation to get
\begin{equation}\label{E_XTYT_1}
\begin{aligned}
\E [\nabla X_T Y_T] = & \E \int_0^T Y_t \bigg[ \partial_x b(t, \theta_t, \theta_{t-\tau})\nabla X_{t} + \partial_{x_{\tau}}b(t, \theta_t, \theta_{t-\tau}) \nabla X_{t-\tau} + \tilde{\E} [\partial_\mu b(t, \theta_t, \theta_{t-\tau}) (\tilde{X}_t) \nabla \tilde{X}_t] \\
	&  + \tilde{\E} [\partial_{\mu_\tau} b(t, \theta_t, \theta_{t-\tau}) (\tilde{X}_{t-\tau}) \nabla \tilde{X}_{t-\tau}] + \partial_{\alpha} b(t, \theta_t, \theta_{t-\tau})\Delta \alpha_\tau + \partial_{\alpha_{\tau}} b(t, \theta_t, \theta_{t-\tau})\Delta \alpha_{t-\tau} \bigg]  dt\\
	& - \E\int_0^T \varphi_t \nabla X_t dt\\
	& + \E \int_0^T Z_t \bigg[\partial_x \sigma(t, \theta_t, \theta_{t-\tau}) \nabla X_t + \partial_{x_{\tau}}\sigma(t, \theta_t, \theta_{t-\tau}) \nabla X_{t-\tau} + \tilde{\E} [\partial_\mu \sigma(t, \theta_t, \theta_{t-\tau}) (\tilde{X_t}) \nabla \tilde{X_t}] \\
	&  + \tilde{\E} [\partial_{\mu_\tau} \sigma(t, \theta_t, \theta_{t-\tau}) (\tilde{X}_{t-\tau}) \nabla \tilde{X}_{t-\tau}] + \partial_{\alpha} \sigma(t, \theta_t, \theta_{t-\tau})\Delta \alpha_t + \partial_{\alpha_\tau} \sigma(t, \theta_t, \theta_{t-\tau})\Delta \alpha_{t-\tau} \bigg] dt
\end{aligned}
\end{equation}
Using the fact that $Y_t = Z_t = 0 $ for $t \in (T, T+\tau]$, we are able to make a change of time, and by Fubini's theorem, so that \eqref{E_XTYT_1} becomes
\begin{equation}\label{E_XTYT_2}
\begin{aligned}
\E [\nabla X_T Y_T] = & \E \int_0^T \bigg(Y_t \partial_x b(t, \theta_t, \theta_{t-\tau}) + Y_{t+\tau} \partial_{x_{\tau}}b(t + \tau, \theta_{t+\tau}, \theta_{t}) + \tilde{\E} [\partial_\mu b(t, \tilde{\theta}_t, \tilde{\theta}_{t-\tau} ) (X_t)] \\
	&  + \tilde{\E} [\partial_{\mu_\tau} b(t + \tau, \tilde{\theta}_{t+\tau}, \tilde{\theta}_{t}) (X_{t}) \bigg) \nabla X_{t} dt \\
	& + \E \int_0^T \bigg(\partial_{\alpha} b(t, \theta_t, \theta_{t-\tau}) + \partial_{\alpha_{\tau}} b(t+\tau, \theta_{t+\tau}, \theta_{t}) \bigg)\Delta \alpha_{t}  dt\\
	& - \E \bigg[\int_0^T \varphi_t \nabla X_t \bigg] dt\\
	& + \E \int_0^T \bigg( Z_{t} \partial_x \sigma(t, \theta_t, \theta_{t-\tau}) + Z_{t + \tau} \partial_{x_{\tau}}\sigma(t + \tau, \theta_{t+\tau}, \theta_{t}) +   \tilde{\E} [\partial_\mu \sigma(t, \tilde{\theta}_t, \tilde{\theta}_{t-\tau}) (X_t)] \\
	& +   \tilde{\E} [\partial_{\mu_\tau} \sigma(t + \tau, \tilde{\theta}_{t+\tau}, \tilde{\theta}_{t}) (X_{t}) \bigg) \nabla X_{t} dt\\
	& + \E \int_0^T \bigg(\partial_{\alpha} \sigma(t, \theta_t, \theta_{t-\tau}) + \partial_{\alpha_{\tau}} \sigma(t+\tau, \theta_{t+\tau}, \theta_{t}) \bigg)\Delta \alpha_{t}  dt\\
\end{aligned}
\end{equation}
Now we define the Hamiltonian $H$ for $(t, x, \mu, x_\tau, \mu_\tau, y, z, \alpha, \alpha_\tau) \in [0,T] \times \R \times \cP_2(\R) \times \R \times \cP_2(\R) \times \R \times \R  \times A \times A  $ as
\begin{multline}\label{Hamiltonian}
H(t, x, \mu, x_\tau, \mu_{\tau}, \alpha, \alpha_\tau, y, z) = b(t, x, \mu, x_\tau, \mu_\tau, \alpha, \alpha_\tau) y + \sigma(t, x, \mu, x_\tau, \mu_\tau, \alpha, \alpha_\tau) z\\ + f (t, x, \mu, x_\tau, \mu_\tau, \alpha)
\end{multline}
 Using the terminal condition of $Y_T$, and plugging \eqref{E_XTYT_2} into \eqref{j-diff}, and setting the integrand containing $\nabla X_t$ to zero, we are able to obtain the adjoint equation is of the following form
 \begin{equation}\label{adjoint}
 \begin{aligned}
 dY_t = & - \bigg\{\partial_x H(t, X_t, \mu_t, X_{t-\tau}, \mu_{t-\tau}, \alpha_t, \alpha_{t-\tau}, Y_t, Z_t) 
     + \tilde{\E} [\partial_\mu H(t, \tilde{X}_t, \mu_t, \tilde{X}_{t-\tau}, \mu_{t-\tau}, \tilde{\alpha}_t, \tilde{\alpha}_{t-\tau}, \tilde{Y}_t, \tilde{Z}_t)(X_t)] \\
     & + \E[\partial_{x_\tau} H(t+\tau, X_{t+\tau}, \mu_{t+\tau}, X_t, \mu_t, \alpha_{t+\tau}, \alpha_t, Y_{t+\tau}, Z_{t+\tau}) | \F_t]\\
  & + \E [ \tilde{\E} [\partial_{\mu_\tau} H(t+\tau, \tilde{X}_{t+\tau}, \mu_{t+\tau}, \tilde{X}_{t}, \mu_{t}, \tilde{\alpha}_t, \tilde{\alpha}_{t-\tau}, \tilde{Y}_{t+\tau}, \tilde{Z}_{t+\tau})(X_{t+\tau})] | \F_t] \bigg\}dt + Z_t dW_t\\
  Y_T = & \partial_x g(X_T, \mu_T) + \tilde{\E}[\partial_\mu g (\tilde{X}_T, \mu_T) (X_T)].\\
 \end{aligned}
 \end{equation}

\begin{Theorem}\label{necessary}
 Let $(\alpha_t)_{0 \leq t \leq T} \in \A$ be optimal, $(X_t)_{0 \leq t \leq T}$ be the associated controlled state, and $(Y_t, Z_t)_{0 \leq t \leq T}$ be the associated adjoint processes defined in \eqref{adjoint}. For any $\beta \in A$, and $t \in [0,T]$,
\begin{multline}\label{nec}
\bigg( \partial_{\alpha} H(t, X_t, \mu_t, X_{t-\tau}, \mu_{t-\tau}, \alpha_t, \alpha_{t-\tau}, Y_t, Z_t) \\
+ \E[ \partial_{\alpha_\tau} H(t+ \tau, X_{t+\tau}, \mu_{t+\tau}, X_{t}, \mu_{t}, \alpha_{t+\tau}, \alpha_{t}, Y_{t+\tau}, Z_{t+\tau}) | \F_t] \bigg) (\beta - \alpha_t) \geq 0 \mbox{ a.e.} \\
\end{multline}
\end{Theorem}

\begin{proof}
Given any $(\beta_t)_{0 \leq t \leq T} \in \A$, we perturbate $\alpha_t$ by $\e(\beta_t - \alpha_t)$ and we define  $\alpha_t^{\e} := \alpha_t + \e (\beta_t - \alpha_t)$ for $0 \leq \e \leq 1$.  Using the adjoint process \eqref{adjoint}, and apply integration by parts formula to $(\nabla X_t Y_t)$. Then plug the result into \eqref{j-diff}, and the Hamiltonian $H$ is defined in \eqref{Hamiltonian}. Also, since  $\alpha$ is optimal, we have
\begin{equation}
\begin{aligned}
0 \leq  & \lim_{\e \to 0} \frac{J(\alpha^\e) - J(\alpha)}{\e}\\
= & \E \int_0^T \bigg([\partial_{\alpha} H(t, \theta_t, \theta_{t-\tau}, Y_t, Z_t) + \E[\partial_{\alpha_\tau}H(t+\tau, \theta_{t+\tau}, \theta_{t}, Y_{t+\tau}, Z_{t+\tau}) | \F_t] \bigg) (\beta_t - \alpha_t)dt
\end{aligned}
\end{equation}
Now, let $C \in \F_t$ be an arbitrary progressively measurable set, and denote $C'$ the complement of $C$. We choose $\beta_t$ to be $\beta_t:= \beta \1_C + \alpha_t \1_{C'}$ for any given $\beta \in A$. Then, 
\begin{equation}
\begin{aligned}
\E \int_0^T \bigg([\partial_{\alpha} H(t, \theta_t, \theta_{t-\tau}, Y_t, Z_t) + \E[\partial_{\alpha_\tau}H(t+\tau, \theta_{t+\tau}, \theta_{t}, Y_{t+\tau}, Z_{t+\tau}) | \F_t] \bigg) (\beta_t - \alpha_t) \1_C dt \geq 0,
\end{aligned}
\end{equation}
which implies, 
\begin{equation}
(\partial_\alpha H(t, \theta_t, \theta_{t-\tau}, Y_t, Z_t) + \E[\partial_{\alpha_\tau}H(t+\tau, \theta_{t+\tau}, \theta_{t}, Y_{t+\tau}, Z_{t+\tau}) | \F_t] ) (\beta - \alpha_t) \geq 0. \mbox{ a.e. }
\end{equation}
\end{proof}

\begin{Remark}
When we further assume that $H$ is convex in $(\alpha_t, \alpha_{t-\tau})$, then for any $\beta, \beta_\tau \in A$ in Theorem \ref{necessary}, we have 
$$ H(t, X_t, \mu_t, X_{t-\tau}, \mu_{t-\tau}, \alpha_t, \alpha_{t-\tau}, Y_t, Z_t) \leq H(t, X_t, \mu_t, X_{t-\tau}, \mu_{t-\tau}, \beta, \beta_{\tau}, Y_t, Z_t), \mbox{ a.e.} $$
as a direct consequence of \eqref{nec}.
\end{Remark}

\begin{Theorem}
Let $(\alpha_t)_{0 \leq t \leq T} \in \A$ be an admissible control. Let $(X_t)_{0 \leq t \leq T}$ be the controlled state, and $(Y_t, Z_t)_{0 \leq t \leq T}$ be the corresponding adjoint processes. We further assume that for each $t$, given $Y_t$ and $Z_t$, the function $(x, \mu, x_{\tau}, \mu_\tau, \alpha, \alpha_{\tau}) \to H(t, x, \mu, x_{\tau}, \mu_\tau, \alpha, \alpha_{\tau}, Y_t, Z_t)$, and the function $(x, \mu) \to g(x, \mu)$ are convex. If 
\begin{equation}\label{maxh}
\begin{aligned}
H(t, X_t, \mu_t, X_{t-\tau}, \mu_{t-\tau}, \alpha_t, \alpha_{t-\tau}, Y_t, Z_t) 
= \inf_{\alpha' \in \A} H(t, X_t, \mu_t, X_{t-\tau}, \mu_{t-\tau}, \alpha'_t, \alpha'_{t-\tau}, Y_t, Z_t) ,
\end{aligned}
\end{equation}
for all $t$, then $(\alpha_t)_{0 \leq t \leq T}$ is an optimal control.

\end{Theorem}

\begin{proof}
Let $(\alpha_t')_{0 \leq t \leq T} \in \A$ be a admissible control, and let $(X'_t)_{0 \leq t \leq T} = (X^{\alpha'}_t)_{_{0 \leq t \leq T}}$ be the corresponding controlled state. From the definition of the objective function as in \eqref{MV_obj}, we first use convexity of $g$, and the terminal condition of the adjoint process $Y_t$ in \eqref{adjoint}, then use the fact that $H$ is convex, and because of \eqref{maxh}, we have the following

\begin{equation}
\begin{aligned}
& J(\alpha) - J(\alpha')\\
= & \E[g(X_T, \mu_T) - g(X_T', \mu_T')] + \E\int_0^T [f(t, \theta_t, X_{t-\tau}, \alpha_{t-\tau}) - f(t, \theta_t', X'_{t-\tau}, \alpha'_{t-\tau}) ] dt\\
\leq & \E[\partial_x g(X_T, \mu_T)(X_T - X_T') + \tilde{\E}[\partial_{\mu} g(X_T, \mu_T) (\tilde{X}_T) (\tilde{X}_T - \tilde{X}_T')]]\\
    & + \E\int_0^T [f(t, \theta_t, X_{t-\tau}, \alpha_{t-\tau}) - f(t, \theta_t', X'_{t-\tau}, \alpha'_{t-\tau})] dt\\
= & \E[(\partial_x g(X_T, \mu_T) + \tilde{E} [\partial_\mu g(\tilde{X}_T, \mu_T) (X_T)]) (X_T - X_T')]\\
    & + \E\int_0^T [f(t, \theta_t, X_{t-\tau}, \alpha_{t-\tau}) - f(t, \theta_t', X'_{t-\tau}, \alpha'_{t-\tau})] dt\\
= & \E[Y_T (X_T - X_T')] + \E\int_0^T [f(t, \theta_t, X_{t-\tau}, \alpha_{t-\tau}) - f(t, \theta_t', X'_{t-\tau}, \alpha'_{t-\tau})] dt\\
= & \E\int_0^T \bigg[ \left(b(t, \theta_t, \theta_{t-\tau}) - b(t, \theta_t', \theta_{t-\tau}') \right)Y_t + \left(\sigma(t, \theta_t, \theta_{t-\tau}) - \sigma(t, \theta_t', \theta_{t-\tau}') \right)Z_t \bigg] dt\\
    & - \E \int_0^T \bigg[ \left(  \partial_x H(t, \theta_{t}, \theta_{t-\tau}, Y_t, Z_t) + \tilde{\E}[\partial_\mu H (t, \tilde{\theta}_t, \tilde{\theta}_{t-\tau}, \tilde{Y}_t, \tilde{Z}_t)(X_t)] \right) (X_t - X_t') \bigg]dt\\
    & -\E \int_0^T \bigg[ \bigg( \E[\partial_{x_\tau}  H(t+\tau, \theta_{t+\tau}, \theta_{t}, Y_{t+\tau}, Z_{t+\tau}) | \F_t ] \\
    & + \E [ \tilde{\E} [\partial_{\mu_\tau} H(t+\tau, \tilde{\theta}_{t+\tau}, \tilde{\theta}_t, \tilde{Y}_{t+\tau}, \tilde{Z}_{t+\tau})(X_{t})] | \F_t] \bigg) (X_t - X_t')\bigg] dt \\
    & + \E \int_0^T \bigg[H(t, \theta_t, \theta_{t-\tau}, Y_t, Z_t) - H(t, \theta_t', \theta_{t-\tau}', Y_t, Z_t) \bigg] dt\\
    &+ \E \int_0^T \bigg[ \left(b(t, \theta_t, \theta_{t-\tau}) - b(t, \theta_t, \theta_{t-\tau}) \right)Y_t + \left(\sigma(t, \theta_t, \theta_{t-\tau}) - \sigma(t, \theta_t, \theta_{t-\tau}) \right)Z_t \bigg] dt\\
\leq & - \E \int_0^T \bigg[\partial_x H (t, \theta_t, \theta_{t-\tau}, Y_t, Z_t)(X_t - X_t') + \tilde{\E} [\partial_\mu H (t, \theta_t, \theta_{t-\tau}, Y_t, Z_t) (\tilde{X}_t) (\tilde{X}_t - \tilde{X_t}')\bigg] dt\\
    & -  \E \int_0^T \bigg[\partial_{x_\tau} H (t, \theta_t, \theta_{t-\tau}, Y_t, Z_t) (X_{t-\tau} - X_{t-\tau}') + \tilde{\E} [\partial_{\mu_\tau} H (t, \theta_t, \theta_{t-\tau}, Y_t, Z_t) (\tilde{X}_{t-\tau}) (\tilde{X}_{t-\tau} - \tilde{X}_{t-\tau}')] \bigg] dt\\
    & + \E \int_0^T \bigg[H (t, \theta_t, \theta_{t-\tau}, Y_t, Z_t) - H (t, \theta_t', \theta_{t-\tau}', Y_t, Z_t) \bigg] dt\\
\leq & \E \int_0^T \bigg[\partial_{\alpha} H (t, \theta_t, \theta_{t-\tau}, Y_t, Z_t)(\alpha_t - \alpha'_t) + \partial_{\alpha_{\tau}} H (t, \theta_t, \theta_{t-\tau}, Y_t, Z_t) (\alpha_{t-\tau} - \alpha_{t-\tau}') \bigg] dt\\
\leq & \E \int_0^T \bigg(\partial_{\alpha} H (t, \theta_t, \theta_{t-\tau}, Y_t, Z_t) + \E[ \partial_{\alpha_\tau}H (t+\tau, \theta_{t+\tau}, \theta_{t}, Y_{t+\tau}, Z_{t+\tau)} | \F_t] \bigg)(\alpha_t - \alpha_t') dt\\
\leq & 0.
\end{aligned}
\end{equation}

\end{proof}

\section{Existence and Uniqueness Result}\label{uniqueness}

Given the necessary and sufficient conditions proven in Section \ref{smp}, we use the optimal control $(\hat{\alpha}_t)_{0\leq t\leq T}$ defined by
\begin{equation}
\begin{aligned}
& \hat{\alpha}(t,X_t, \mu_t, X_{t-\tau}, \mu_{t-\tau}, Y_t, Z_t, \E[Y_{t+\tau} | \F_t], \E[Z_{t+\tau} | \F_t]) \\
    = & \arg \min_{\alpha \in \A}  H(t, X_t, \mu_t, X_{t-\tau}, \mu_{t-\tau}, \alpha_t, \alpha_{t-\tau}, Y_t, Z_t),
\end{aligned}
\end{equation}
to establish the solvability result of the McKean-Vlasov FABSDE \eqref{MV_SDDE} and \eqref{adjoint} for $t \in [0,T]$:
\begin{equation}\label{FABSDE}
\begin{aligned}
dX_t = & b(t, X_t, \mu_t, X_{t-\tau}, \mu_{t-\tau}, \hat{\alpha}_t, \hat{\alpha}_{t-\tau}) dt + \sigma(t, X_t, \mu_t, X_{t-\tau}, \mu_{t-\tau}, \hat{\alpha}_t, \hat{\alpha}_{t-\tau}) dW_t,\\
dY_t = & - \bigg \{\partial_x H(t, X_t, \mu_t, X_{t-\tau}, \mu_{t-\tau}, \hat{\alpha}_t, \hat{\alpha}_{t-\tau}, Y_t, Z_t) 
     + \tilde{\E} [\partial_\mu H(t, \tilde{X}_t, \mu_t,  \tilde{X}_{t-\tau}, \mu_{t-\tau}, \tilde{\hat{\alpha}}_t, \tilde{\hat{\alpha}}_{t-\tau}, \tilde{Y}_t, \tilde{Z}_t)(X_t)] \\
     & + \E[\partial_{x_\tau} H(t+\tau, X_{t+\tau}, \mu_{t+\tau}, X_t, \mu_t, \hat{\alpha}_{t+\tau}, \hat{\alpha}_t, Y_{t+\tau}, Z_{t+\tau}) | \F_t]\\
  & + \E [ \tilde{\E} [\partial_{\mu_\tau} H(t+\tau, \tilde{X}_{t+\tau}, \mu_{t+\tau}, \tilde{X}_t, \mu_t, \tilde{\hat{\alpha}}_{t+\tau}, \tilde{\hat{\alpha}}_{t}, \tilde{Y}_{t+\tau}, \tilde{Z}_{t+\tau})(X_{t})] | \F_t] \bigg\}dt + Z_t dW_t\\
\end{aligned}
\end{equation}
with initial condition $X_0 = x_0;   X_t  = \hat{\alpha}_t  = 0 \mbox{ for } t \in [-\tau, 0)$ and terminal condition $Y_T = \partial_x g(X_T, \mu_T) + \tilde{\E}[\partial_\mu g (\tilde{X}_T, \mu_T) (X_T)].$
In addition to assumption (H 4.1), we further assume
\begin{enumerate}[label = (H5.\arabic*)]
\item The drift and volatility functions $b$ and $\sigma$ are linear in $x, \mu, x_{\tau}, \mu_{\tau}, \alpha, \alpha_{\tau}$. For all $(t, x, \mu, x_\tau, \mu_\tau, \alpha, \alpha_\tau) \in [0, T] \times \R \times \cP_2(\R) \times \cP_2(\R) \times A \times A$, we assume that
\begin{equation}
\begin{aligned}
b(t, x, \mu, x_\tau, \mu_\tau, \alpha, \alpha_\tau) =& b_0(t) + b_1(t)x + \bar{b}_1(t) m + b_2(t) x_\tau + \bar{b}_2(t) m_\tau + b_3(t) \alpha + b_4(t) \alpha_\tau,\\
\sigma(t, x, \mu, x_\tau, \mu_\tau, \alpha, \alpha_\tau) =& \sigma_0(t) + \sigma_1(t)x + \bar{\sigma}_1(t) m + \sigma_2(t) x_\tau + \bar{\sigma}_2(t) m_\tau + \sigma_3(t) \alpha + \sigma_4(t) \alpha_\tau,\\
\end{aligned}
\end{equation}
for some measurable deterministic functions $b_0, b_1, \bar{b}_1, b_2, \bar{b}_2, b_3, b_4, \sigma_0, \sigma_1, \bar{\sigma}_1, \sigma_2, \bar{\sigma}_2, \sigma_3, \sigma_4$ with values in $\R$ bounded by $R$, and we have used the notation $m = \int x d \mu(x)$ and $m_\tau = \int x d \mu_\tau(x)$ for the mean of measures $\mu$ and $\mu_\tau$ respectively.
\item The derivatives of $f$ and $g$ with respect to $(x, x_\tau, \mu, \mu_\tau, \alpha)$ and $(x, \mu)$ are Lipschitz continuous with Lipschitz constant $L$.


\item The function $f$ is strongly L-convex, which means 
that for any $t \in [0,T]$, any $x, x', x_\tau, x_\tau' \in \R$, any $\alpha, \alpha' \in A$, any $\mu, \mu', \mu_\tau, \mu_\tau' \in \cP_2(\R)$,  any random variables $X$ and $X'$ having $\mu$ and $\mu'$ as distribution, and  any random variables $X_\tau$ and $X'_\tau$ having $\mu_\tau$ and $\mu'_\tau$ as distribution, then
\begin{equation}
\begin{aligned}
& f(t, x', \mu', x_\tau', \mu_\tau', \alpha') - f(t, x, \mu, x_\tau, \mu_\tau, \alpha) \\
&  - \partial_x f(t, x, \mu, x_\tau, \mu_\tau, \alpha)(x'-x)
 - \partial_{x_\tau} f(t, x, \mu, x_\tau, \mu_\tau, \alpha)(x_\tau'-x_\tau)\\
&  - \E[\partial_\mu f(t, x, \mu, x_\tau, \mu_\tau, \alpha)(X) \cdot (X' - X)]
 - \E[\partial_{\mu_\tau} f(t, x, \mu, x_\tau, \mu_\tau, \alpha)(X_\tau) \cdot (X_\tau' - X_\tau)]\\
&  - \partial_\alpha f(t, x, \mu, x_\tau, \mu_\tau, \alpha)(\alpha'-\alpha)
\geq \kappa |\alpha' - \alpha|^2.
\end{aligned}
\end{equation}
The function $g$ is also assumed to be $L$-convex in $(x, \mu)$.
\end{enumerate}

\begin{Theorem}\label{existence}
Under assumptions (H5.1-H5.3), the McKean-Vlasov FABSDE \eqref{FABSDE} is uniquely solvable.
\end{Theorem}

The proof is based on continuation methods. Let $\lambda \in [0, 1]$, consider the following class of  McKean-Vlasov FABSDEs, denoted by MV-FABSDE($\lambda$), for $t \in [0, T]$:
\begin{equation}\label{FABSDE_lambda}
\begin{aligned}
dX_t = & (\lambda b(t, \theta_t, \theta_{t-\tau})  + I_t^b )dt + (\lambda \sigma(t, \theta_t, \theta_{t-\tau}) + I_t^{\sigma}) dW_t,\\
dY_t = & - \bigg\{ \lambda \bigg(\partial_x H(t, \theta_t, \theta_{t-\tau}, Y_t, Z_t) 
     + \tilde{\E} [\partial_\mu H(t, \tilde{\theta}_t, \tilde{\theta}_{t-\tau}, \tilde{Y}_t, \tilde{Z}_t)(X_t)] \\
     & + \E[\partial_{x_\tau} H(t+\tau, \theta_{t+\tau}, \theta_{t}, Y_{t+\tau}, Z_{t+\tau}) | \F_t]
      + \E [ \tilde{\E} [\partial_{\mu_\tau} H(t+\tau, \tilde{\theta}_{t+\tau}, \tilde{\theta}_{t}, \tilde{Y}_{t+\tau}, \tilde{Z}_{t+\tau})(X_{t})] | \F_t] \bigg) + I_t^f \bigg\} dt \\
      & + Z_t dW_t, \\
\end{aligned}
\end{equation}
where we denote $\theta_t = (X_t, \mu_t, \alpha_t)$, with optimality condition $$\alpha_t = \hat{\alpha}(t, X_t, \mu_t, X_{t-\tau}, \mu_{t-\tau}, Y_t, Z_t, \E[Y_{t+\tau} | \F_t], \E[Z_{t+\tau} | \F_t]), t \in [0,T],$$
and with initial condition $X_0 = x_0;   X_t  = \alpha_t  = 0 \mbox{ for } t \in [-\tau, 0)$ and terminal condition $$Y_T =  \lambda \bigg\{ \partial_x g(X_T, \mu_T) + \tilde{\E}[\partial_\mu g (\tilde{X}_T, \mu_T) (X_T) \bigg\} + I_T^g,$$ and $Y_t = 0$ for $t \in (T, T+\tau]$,
where $(I^b_t, I^{\sigma}_t, I^{f}_t)_{0 \leq t \leq T}$ are some square-integrable progressively measurable processes with values in $\R$, and $I_T^g \in L^2(\W, \F_T, \P)$ is a square integrable $\F_T$-measurable random variable with value in $\R$. 

Observe that when $\lambda = 0$, system \eqref{FABSDE_lambda} becomes decoupled standard SDE and BSDE, which has an unique solution. When setting $\lambda = 1, \ I_t^b = I_t^{\sigma} = I_t^f = 0$ for $0 \leq t \leq T$, and $I_T^g = 0$, we are able to recover the system of \eqref{FABSDE}. 

\begin{Lemma}\label{lemma}
Given $\lambda_0 \in [0,1)$, for any square-integrable progressively measurable processes $(I^b_t, I^{\sigma}_t, I^{f}_t)_{0 \leq t \leq T}$, and $I_T^g \in L^2(\W, \F_T, \P)$, such that system FABSDE($\lambda_0$) admits a unique solution, then there exists $\delta_0 \in (0,1)$, which is independent on $\lambda_0$, such that the system MV-FABSDE($\lambda$) admits a unique solution for any $\lambda \in [\lambda_0, \lambda_0 + \delta_0]$.
\end{Lemma}

\begin{proof}
Assuming that $(\check{X}, \check{Y}, \check{Z}, \check{\alpha})$ are given as an input, for any $\lambda \in [\lambda_0, \lambda_0 + \delta_0]$, where $\delta_0 > 0$ to be determined, denoting $\delta := \lambda - \lambda_0 \leq \delta_0$, we take
\begin{equation}
\begin{aligned}
    I_t^b \leftarrow & \delta [b(t, \check{\theta}_t, \check{\theta}_{t-\tau})] + I_t^b, \\
    I_t^\sigma \leftarrow & \delta [\sigma(t, \check{\theta}_t, \check{\theta}_{t-\tau})] + I_t^\sigma, \\
    I_t^f \leftarrow &  \delta\bigg[\partial_x H(t, \check{\theta}_t, \check{\theta}_{t-\tau}, Y_t, Z_t) 
     + \tilde{\E} [\partial_\mu H(t, \tilde{\check{\theta}}_t, \tilde{\check{\theta}}_{t-\tau}, \tilde{\check{Y}}_t, \tilde{\check{Z}}_t)(X_t)] \\
     & + \E[\partial_{x_\tau} H(t+\tau, \check{\theta}_{t+\tau}, \check{\theta}_{t}, \check{Y}_{t+\tau}, \check{Z}_{t+\tau}) | \F_t]
      + \E [ \tilde{\E} [\partial_{\mu_\tau} H(t+\tau, \tilde{\check{\theta}}_{t+\tau}, \tilde{\check{\theta}}_{t}, \tilde{\check{Y}}_{t+\tau}, \tilde{\check{Z}}_{t+\tau})(\check{X}_{t})] | \F_t] \bigg] \\
      & + I_t^f, \\
    I_T^g \leftarrow & \delta \bigg[ \partial_x g(\check{X}_T, \mu_T) + \tilde{\E}[\partial_\mu g (\tilde{\check{X}}_T, \mu_T) (\check{X}_T) \bigg] + I_T^g.
\end{aligned}
\end{equation}
According to the assumption, let $(X, Y, Z)$ be the solutions of MV-FABSDE($\lambda_0$) corresponding to inputs $(\check{X}, \check{Y}, \check{Z})$, i.e., for $t \in [0, T]$
\begin{equation}
\begin{aligned}
        d X_t =&  (\lambda_0  b_t + \delta \check{b}_t + I_t^b) dt + (\lambda_0  \sigma_t + \delta  \check{\sigma}_t +I_t^{\sigma}) dW_t, \\
         d Y_t = & - \bigg \{ \lambda_0(\partial_x H_t +   \tilde{\E}[\partial_{\mu} \tilde{H}_t(X_t)] 
         +  \E[\partial_{x_\tau} H_{t+\tau} | \F_t] 
         +  \E [\tilde{\E} [\partial_{\mu_\tau} \tilde{H}_{t+\tau}(X_t)] |\F_t] ) \\
      & +  \delta (\partial_x \check{H}_t +   \tilde{\E}[ \partial_{\mu} \tilde{\check{H}}_t(\check{X}_t)] +  \E[\partial_{x_\tau} \check{H}_{t+\tau} | \F_t] + \E [\tilde{\E} [\partial_{\mu_\tau} \tilde{\check{H}}_{t+\tau}(\check{X}_t)] |\F_t] )+ I_t^f \bigg \} dt \\
      & + Z_t dW_t,
\end{aligned}
\end{equation}
with initial condition, $X_0 = x_0$, $X_s = \alpha_s = 0$ for $s \in [-\tau, 0)$, and terminal condition
\begin{equation}\label{terminal}
 Y_T =  \lambda_0 \bigg( \partial_x g_T + \tilde{\E}[\partial_\mu \tilde{g}_T(X_T)] \bigg) 
 + \delta \bigg( \partial_x \check{g}_T + \tilde{\E}[\partial_\mu \tilde{\check{g}}_T (\check{X}_T)] \bigg) + I_T^g, 
\end{equation}
and $ Y_t = Z_t = 0$ for $t \in (T, T+\tau]$, where we have used simplified notations,
\begin{equation}
\begin{aligned}
    b_t  := & b(t, \theta_t, \theta_{t-\tau}); \quad 
    \check{b }_t  :=  b(t, \check{\theta}_t, \check{\theta}_{t-\tau}) ;  \quad
     \sigma_t  :=  \sigma(t, \theta_t, \theta_{t-\tau}); \quad
    \check{\sigma}_t  :=  \sigma(t, \check{\theta}_t, \check{\theta}_{t-\tau});\\
    \partial_x H_t := & \partial_x H(t, \theta_t, \theta_{t-\tau}, Y_t, Z_t) ; \quad
    \tilde{\E} [\partial_{\mu} \tilde{H}_t(X_t)] :=  \tilde{\E} [\partial_\mu H(t, \tilde{\theta}_t, \tilde{\theta}_{t-\tau}, \tilde{Y}_t, \tilde{Z}_t)(X_t)]\\
     \E[\partial_{x_\tau} H_{t+\tau} | \F_t] := & \E[\partial_{x_\tau} H(t+\tau, \theta_{t+\tau}, \theta_{t}, Y_{t+\tau}, Z_{t+\tau}) | \F_t]; \\
    \E[ \tilde{\E} [\partial_{\mu} \tilde{H}_t(X_t)] | \F_t] := & \E [ \tilde{\E} [\partial_{\mu_\tau} H(t+\tau, \tilde{\theta}_{t+\tau}, \tilde{\theta}_{t}, \tilde{Y}_{t+\tau}, \tilde{Z}_{t+\tau})(X_{t})] | \F_t]; \\
    \partial_x g_T := & \partial_x g(X_T, \mu_T); \quad
     \tilde{\E}[\partial_\mu \tilde{g}_T(X_T)] ] :=  \tilde{\E}[\partial_\mu g(\tilde{X}_T, \mu_T) (X_T)] \\
     \mbox{similar notation for } & \partial_x \check{H}_t, 
     \ \tilde{\E}[ \partial_{\mu} \tilde{\check{H}}_t(\check{X}_t)], 
     \ \E[\partial_{x_\tau} \check{H}_{t+\tau} | \F_t], 
     \ \mbox{ and }  \E [\tilde{\E} [\partial_{\mu_\tau} \tilde{\check{H}}_{t+\tau}(\check{X}_t)] |\F_t].
\end{aligned}
\end{equation}
We would like to show that the map $\Phi: (\check{X}, \check{Y}, \check{Z}, \check{\alpha}) \to \Phi(\check{X}, \check{Y}, \check{Z}, \check{\alpha}) = (X, Y, Z, \alpha)$ is a contraction.
Consider $(\Delta X, \Delta Y, \Delta Z, \Delta \alpha) = (X - X', Y - Y', Z - Z', \alpha - \alpha')$, where $(X', Y', Z', \alpha') = \Phi(\check{X}', \check{Y}', \check{Z}', \check{\alpha}')$. 
In addition, for the following computation, we have used simplified notation:
\begin{equation}
\begin{aligned}
    \Delta b_t  := & b(t, \theta_t, \theta_{t-\tau}) - b(t, \theta_t',\theta_{t-\tau}'); \quad 
    \Delta \check{b }_t  :=  b(t, \check{\theta}_t, \check{\theta}_{t-\tau}) - b(t, \check{\theta}_t', \check{\theta}_{t-\tau}'); \\
    \Delta \sigma_t  := & \sigma(t, \theta_t, \theta_{t-\tau}) - \sigma(t, \theta_t',\theta_{t-\tau}'); \quad 
    \Delta \check{\sigma}_t  :=  \sigma(t, \check{\theta}_t, \check{\theta}_{t-\tau}) - \sigma(t, \check{\theta}_t', \check{\theta}_{t-\tau}')\\
     \partial_x g_T := & \partial_x g(X_T, \mu_T) - \partial_x g(X'_T, \mu_T); \\
      \Delta \tilde{\E}[\partial_\mu \tilde{g}_T(X_T)] ] := & \tilde{\E}[\partial_\mu g(\tilde{X}_T, \mu_T) (X_T)] - \tilde{\E}[\partial_\mu g(\tilde{\check{X}}'_T, \mu_T) (X_T')]\\
    \Delta \partial_x H_t := & \partial_x H(t, \theta_t, \theta_{t-\tau}, Y_t, Z_t) - \partial_x H(t, \theta_t', \theta_{t-\tau}', Y_t, Z_t)\\
    \Delta \tilde{\E} [\partial_{\mu} \tilde{H}_t(X_t)] := & \tilde{\E} [\partial_\mu H(t, \tilde{\theta}_t, \tilde{\theta}_{t-\tau}, \tilde{Y}_t, \tilde{Z}_t)(X_t)] - \tilde{\E} [\partial_\mu H(t, \tilde{\theta}'_t, \tilde{\theta}'_{t-\tau}, \tilde{Y}_t, \tilde{Z}_t)(X_t')] \\
    \Delta \E[\partial_{x_\tau} H_{t+\tau} | \F_t] := & \E[\partial_{x_\tau} H(t+\tau, \theta_{t+\tau}, \theta_{t}, Y_{t+\tau}, Z_{t+\tau}) | \F_t] - \E[\partial_{x_\tau} H(t+\tau, \theta'_{t+\tau}, \theta'_{t}, Y_{t+\tau}, Z_{t+\tau}) | \F_t]\\
    \Delta \E[ \tilde{\E} [\partial_{\mu} \tilde{H}_t(X_t)] | \F_t] := & \E [ \tilde{\E} [\partial_{\mu_\tau} H(t+\tau, \tilde{\theta}_{t+\tau}, \tilde{\theta}_{t}, \tilde{Y}_{t+\tau}, \tilde{Z}_{t+\tau})(X_{t})] | \F_t]\\
    & - \E [ \tilde{\E} [\partial_{\mu_\tau} H(t+\tau, \tilde{\theta}'_{t+\tau}, \tilde{\theta}_{t}', \tilde{Y}_{t+\tau}, \tilde{Z}_{t+\tau})(X_{t}')] | \F_t]  \\
     \mbox{similar notation for } &  \Delta \partial_x \check{H}_t, 
     \ \Delta \tilde{\E}[ \partial_{\mu} \tilde{\check{H}}_t(\check{X}_t)], 
     \ \Delta \E[\partial_{x_\tau} \check{H}_{t+\tau} | \F_t], 
     \ \mbox{ and } \Delta \E [\tilde{\E} [\partial_{\mu_\tau} \tilde{\check{H}}_{t+\tau}(\check{X}_t)] |\F_t].
\end{aligned}
\end{equation}

Applying integration by parts to $\Delta X_t Y_t$, we have
\begin{equation}
\begin{aligned}
& d(\Delta X_t Y_t)\\
= & Y_t \bigg\{[\lambda_0 \Delta b_t + \delta \Delta \check{b }_t ] dt + [\lambda_0 \Delta \sigma_t + \delta \Delta \check{\sigma}_t ] dW_t \bigg\}\\
	&- \Delta X_t \bigg \{ \lambda_0(\partial_x H_t + \tilde{\E}[\partial_{\mu} \tilde{H}_t(X_t)] 
         + \E[\partial_{x_\tau} H_{t+\tau} | \F_t] 
         + \E [\tilde{\E} [\partial_{\mu_\tau} \tilde{H}_{t+\tau}(X_t)] |\F_t] ) \\
      & +  \delta (\partial_x \check{H}_t + \tilde{\E}[ \partial_{\mu} \tilde{\check{H}}_t(\check{X}_t)] +  \E[\partial_{x_\tau} \check{H}_{t+\tau} | \F_t] +  \E [\tilde{\E} [\partial_{\mu_\tau} \tilde{\check{H}}_{t+\tau}(\check{X}_t)] |\F_t] ) \bigg \} dt \\
      & + \Delta X_t Z_t dW_t + (\lambda_0 \Delta \sigma_t + \delta \Delta \check{\sigma}_t) Z_t dt.
\end{aligned}
\end{equation}
After integrating from $0$ to $T$, and taking expectation on both sides, we obtain
\begin{equation}\label{deltaXY1}
\begin{aligned}
& \E[\Delta X_T Y_T]\\
= & \lambda_0 \E \int_0^T \bigg(\Delta b_t Y_t + \Delta \sigma_t Z_t - \Delta X_t (\partial_x H_t + \tilde{\E}[\partial_{\mu} \tilde{H}_t(X_t)] 
         + \E[\partial_{x_\tau} H_{t+\tau} | \F_t] \\
         & + \E [\tilde{\E} [\partial_{\mu_\tau} \tilde{H}_{t+\tau}(X_t)] |\F_t] )\bigg) dt \\
       &+  \delta \E \int_0^T \bigg(\Delta \check{b}_t Y_t + \Delta \check{\sigma} Z_t - \Delta X_t (\partial_x \check{H}_t + \tilde{\E}[\partial_{\mu} \tilde{\check{H}}_t(\check{X}_t)] 
         + \E[\partial_{x_\tau} \check{H}_{t+\tau} | \F_t] \\
         & + \E [\tilde{\E} [\partial_{\mu_\tau} \tilde{\check{H}}_{t+\tau}(\check{X}_t)] |\F_t] )\bigg) dt \\  
\end{aligned}
\end{equation}
In the meantime, from the terminal condition of $Y_T$ given in \eqref{terminal}, and since $g$ is convex, we also have 
\begin{equation}\label{deltaXY2}
\begin{aligned}
& \E[\Delta X_T Y_T]\\
= & \E\left[\Delta X_T  \bigg(\lambda_0 ( \partial_x g_T + \tilde{\E}[\partial_\mu \tilde{g}_T(X_T)] ) + \delta ( \partial_x \check{g}_T + \tilde{\E}[\partial_\mu \tilde{\check{g}}_T (\check{X}_T)] ) + I_T^g \bigg)\right]\\
\geq &  \lambda_0 \E[g(X_T, \mu_T ) - g(X_T' - \mu_T^')] 
	+ \delta \Delta X_T ( \partial_x \check{g}_T + \tilde{\E}[\partial_\mu \tilde{\check{g}}_T (\check{X}_T)]) + \Delta X_T I_T^g\\
\end{aligned}
\end{equation}
Following the proof of sufficient part of maximum principle and using \eqref{deltaXY1}, and \eqref{deltaXY2}, we find
\begin{equation}\label{lambdaJ1}
\begin{aligned}
& \lambda_0(J(\alpha) - J(\alpha'))\\
		= & \lambda_0 \E[g(X_T, \mu_T) - g(X_T', \mu_T')] + \lambda_0 \E\int_0^T [f(t, \theta_t, X_{t-\tau}, \mu_{t-\tau}) - f(t, \theta_t', X'_{t-\tau}, \mu_{t-\tau}')] dt\\
		   \leq & \E[\Delta X_T Y_T] - \delta \Delta X_T ( \partial_x \check{g}_T + \tilde{\E}[\partial_\mu \tilde{\check{g}}_T (\check{X}_T)]) - \Delta X_T I_T^g\\
    & + \lambda_0 \E\int_0^T [f(t, \theta_t, X_{t-\tau}, \mu_{t-\tau}) - f(t, \theta_t', X'_{t-\tau}, \mu_{t-\tau}')] dt\\
	= & \lambda_0 \E \int_0^T \bigg[\Delta b_t Y_t + \Delta \sigma_t Z_t - \Delta X_t (\partial_x H_t + \tilde{\E}[\partial_{\mu} \tilde{H}_t(X_t)] 
         + \E[\partial_{x_\tau} H_{t+\tau} | \F_t] \\
         & + \E [\tilde{\E} [\partial_{\mu_\tau} \tilde{H}_{t+\tau}(X_t)] |\F_t] )\bigg] dt \\
       &+  \delta \E \int_0^T \bigg[\Delta \check{b}_t Y_t + \Delta \check{\sigma}_t Z_t - \Delta X_t (\partial_x \check{H}_t + \tilde{\E}[\partial_{\mu} \tilde{\check{H}}_t(\check{X}_t)] 
         + \E[\partial_{x_\tau} \check{H}_{t+\tau} | \F_t] \\
         & + \E [\tilde{\E} [\partial_{\mu_\tau} \tilde{\check{H}}_{t+\tau}(X_t)] |\F_t] )\bigg] dt \\ 
         & + \lambda_0 \E \int_0^T [H(t, \theta_t, \theta_{t-\tau}, Y_t, Z_t) - H(t, \theta'_t, \theta'_{t-\tau}, Y_t, Z_t)] dt \\
         & - \lambda_0 \E \int_0^T (\Delta b_t Y_t + \Delta_t \sigma Z_t )dt - \delta \Delta X_T ( \partial_x \check{g}_T + \tilde{\E}[\partial_\mu \tilde{\check{g}}_T (\check{X}_T)]) - \Delta X_T I_T^g\\
     = & \lambda_0 \E \int_0^T  \bigg[H(t, \theta_t, \theta_{t-\tau}, Y_t, Z_t) - H(t, \theta'_t, \theta'_{t-\tau}, Y_t, Z_t) - \Delta X_t (\partial_x H_t + \tilde{\E}[\partial_{\mu} \tilde{H}_t(X_t)] 
         + \E[\partial_{x_\tau} H_{t+\tau} | \F_t] \\
         & + \E [\tilde{\E} [\partial_{\mu_\tau} \tilde{H}_{t+\tau}(X_t)] |\F_t] )\bigg] dt 
         +  \delta \E \int_0^T \bigg[\Delta \check{b}_t Y_t + \Delta \check{\sigma} Z_t - \Delta X_t (\partial_x \check{H}_t + \tilde{\E}[\partial_{\mu} \tilde{\check{H}}_t(\check{X}_t)] 
         + \E[\partial_{x_\tau} \check{H}_{t+\tau} | \F_t] \\
         & + \E [\tilde{\E} [\partial_{\mu_\tau} \tilde{\check{H}}_{t+\tau}(\check{X}_t)] |\F_t] )\bigg] dt 
         - \delta \Delta X_T ( \partial_x \check{g}_T + \tilde{\E}[\partial_\mu \tilde{\check{g}}_T (\check{X}_T)]) - \Delta X_T I_T^g\\
     \leq & -\E \int_0^T \lambda_0 \kappa | \Delta \alpha_t|^2 dt
     		+  \delta \E \int_0^T \bigg[\Delta \check{b}_t Y_t + \Delta \check{\sigma}_t Z_t - \Delta X_t (\partial_x \check{H}_t + \tilde{\E}[\partial_{\mu} \tilde{\check{H}}_t(\check{X}_t)] 
         + \E[\partial_{x_\tau} \check{H}_{t+\tau} | \F_t] \\
         & + \E [\tilde{\E} [\partial_{\mu_\tau} \tilde{\check{H}}_{t+\tau}(\check{X}_t)] |\F_t] )\bigg] dt 
         - \delta \Delta X_T ( \partial_x \check{g}_T + \tilde{\E}[\partial_\mu \tilde{\check{g}}_T (\check{X}_T)]) - \Delta X_T I_T^g\\
\end{aligned}
\end{equation}
Reverse the role of $\alpha$ and $\alpha'$, we also have 
\begin{equation}\label{lambdaJ2}
\begin{aligned}
& \lambda_0(J(\alpha') - J(\alpha))\\
     \leq & -\E \int_0^T \lambda_0 \kappa | \Delta \alpha'_t|^2 dt
     		+  \delta \E \int_0^T \bigg[\Delta \check{b}'_t Y'_t + \Delta \check{\sigma}'_t Z'_t - \Delta X'_t (\partial_x \check{H}'_t + \tilde{\E}[\partial_{\mu} \tilde{\check{H}}'_t(\check{X}'_t)] 
         + \E[\partial_{x_\tau} \check{H}'_{t+\tau} | \F_t] \\
         & + \E [\tilde{\E} [\partial_{\mu_\tau} \tilde{\check{H}}'_{t+\tau}(\check{X}'_t)] |\F_t] )\bigg] dt 
         - \delta \Delta X'_T ( \partial_x \check{g}_T + \tilde{\E}[\partial_\mu \tilde{\check{g}}_T (\check{X}'_T)]) - \Delta X'_T I_T^g\\
\end{aligned}
\end{equation}

Summing \eqref{lambdaJ1} and \eqref{lambdaJ2}, using the fact that $b$ and $\sigma$ have the linear form, using change of time and Lipschitz assumption, it yields
\begin{equation}\label{alpha}
\begin{aligned}
& 2 \lambda_0 \kappa \E \int_0^T |\Delta \alpha_t|^2 dt \\
    \leq &  \delta \E \int_0^T \bigg[\Delta \check{b}_t \Delta Y_t + \Delta \check{\sigma} \Delta Z_t - \Delta X_t (\Delta \partial_x \check{H}_t + \Delta \tilde{\E}[\partial_{\mu} \tilde{\check{H}}_t(X_t)] 
         + \Delta \E[\partial_{x_\tau} \check{H}_{t+\tau} | \F_t] \\
         & + \Delta \E [\tilde{\E} [\partial_{\mu_\tau} \tilde{\check{H}}_{t+\tau}(X_t)] |\F_t] )\bigg] dt 
         + \delta \Delta X_T ( \partial_{x'} \check{g}'_T  - \partial_x \check{g}_T +\tilde{\E}[\partial_\mu \tilde{\check{g}}'_T (\check{X}_T)] -  \tilde{\E}[\partial_\mu \tilde{\check{g}}_T (\check{X}_T)] )\\
     \leq & \frac{1}{2}  \E \int_0^T \bigg[\e (|\Delta X_t|^2  + |\Delta Y_t|^2 + |\Delta Z_t|^2) + \frac{1}{\e} \delta^2 \bigg( |\Delta \check{b}_t|^2 + |\Delta \check{\sigma}|^2 \\
         & + |\Delta \partial_x \check{H}_t + \Delta \tilde{\E}[\partial_{\mu} \tilde{\check{H}}_t(X_t)] 
         + \Delta \E[\partial_{x_\tau} \check{H}_{t+\tau} | \F_t] 
         + \Delta \E [\tilde{\E} [\partial_{\mu_\tau} \tilde{\check{H}}_{t+\tau}(X_t)] |\F_t] |^2 \bigg) \bigg] dt\\
         &+ \frac{1}{2} \bigg(\e |\Delta X_T|^2 +  \frac{1}{\e} \delta^2 \bigg| \partial_{x'} \check{g}'_T  - \partial_x \check{g}_T +\tilde{\E}[\partial_\mu \tilde{\check{g}}'_T (\check{X}_T)] -  \tilde{\E}[\partial_\mu \tilde{\check{g}}_T (\check{X}_T)]  \bigg|^2 \bigg) \\
     \leq & \frac{1}{2} \e  \E \left[   \int_0^T \e (|\Delta X_t|^2  + |\Delta Y_t|^2 + |\Delta Z_t|^2 + |\Delta \alpha_t|^2) dt + |\Delta X_T|^2 \right] \\
         &+ \frac{1}{2} \delta \frac{C}{\e}  \E \left[ \int_0^T (|\Delta \check{X}_t|^2  + |\Delta \check{Y}_t|^2 + |\Delta \check{Z}_t|^2 + |\Delta \check{\alpha}_t|^2)] dt + |\Delta \check{X}_T|^2 \right], \\
\end{aligned}
\end{equation}

Next, we apply Ito's formula to $\Delta X_t^2$, 
\begin{equation}
\begin{aligned}
& d\Delta X_t^2\\
= & 2 \Delta X_t dX_t + d \langle X, X \rangle_t\\
= & 2 \Delta X_t (\lambda_0 \Delta b_t + \delta \Delta \check{b }_t ) dt + 2 \Delta X_t (\lambda_0 \Delta \sigma_t + \delta \Delta \check{\sigma}_t ) dW_t 
    + \bigg(\lambda_0 \Delta \sigma_t + \delta \Delta \check{\sigma}_t \bigg)^2 dt\\
\end{aligned}
\end{equation}
Then integrate from $0$ to $T$, and take expectation,
\begin{equation}
\begin{aligned}
& \E[|\Delta X_t|^2] \\
= & 2\lambda_0  \E \int_0^t |\Delta X_s \Delta b_s| ds + 2 \delta \E \int_0^t |\Delta X_s \Delta \check{b}_s| ds
     +\E \int_0^t |\lambda_0 \Delta \sigma_s + \delta \Delta \check{\sigma}_s|^2 ds\\
\leq & \lambda_0 \E \int_0^t (|\Delta X_s|^2 + |\Delta b_s|^2) ds + \E \int_0^t (|\Delta X_s|^2 + \delta^2 |\Delta \check{b}_s|^2) ds\\
    &+ \E \int_0^t (2 \lambda_0^2  |\Delta \sigma_s|^2 + 2 \delta^2 |\Delta \check{\sigma}_s|^2) ds\\
\leq & C \E \int_0^{t+\tau} (|\Delta X_s|^2 + |\Delta \alpha_s|^2) ds + \delta C \E \int_0^{t+\tau} (|\Delta \check{X}_s|^2 + |\Delta \check{\alpha}_s|^2) ds
\end{aligned}
\end{equation}
From Gronwall's inequality, we can obtain
\begin{equation}\label{X}
\sup_{0 \leq t \leq T} \E[|X_t|^2] \leq C \E \int_0^T |\Delta \alpha_t|^2 dt + \delta C \E\int_0^T (|\Delta \check{X}_t|^2 + |\Delta \check{\alpha}_t|^2) dt
\end{equation}
Similarly, applying Ito's formula to $|\Delta Y_t|^2$, and taking expectation, we have
\begin{equation}
\begin{aligned}
& \E\left[|\Delta Y_t|^2 + \int_t^T |\Delta Z_s|^2 ds \right] \\
=& 2 \lambda_0 \E \int_t^T \bigg|\Delta Y_t \bigg(\Delta \partial_x H_t +  \Delta \tilde{\E}[\partial_{\mu} \tilde{H}_t(X_t)] 
         + \Delta \E[\partial_{x_\tau} H_{t+\tau} | \F_t] 
         + \Delta \E [\tilde{\E} [\partial_{\mu_\tau} \tilde{H}_{t+\tau}(X_t)] |\F_t] \bigg )   \bigg| \\
      & +  2 \delta  \E \int_t^T \bigg|\Delta Y_t  \bigg( \Delta \partial_x \check{H}_t +  \Delta \tilde{\E}[ \partial_{\mu} \tilde{\check{H}}_t(\check{X}_t)] + \Delta \E[\partial_{x_\tau} \check{H}_{t+\tau} | \F_t] + \Delta \E [\tilde{\E} [\partial_{\mu_\tau} \tilde{\check{H}}_{t+\tau}(\check{X}_t)] |\F_t] \bigg) \bigg|\\
      & + \E |\Delta Y_T|^2\\
\leq & \E \int_t^T \bigg(  \frac{1}{\e} |\Delta Y_t|^2 + \e \lambda_0^2 \bigg|\Delta \partial_x H_t +  \Delta \tilde{\E}[\partial_{\mu} \tilde{H}_t(X_t)] 
         + \Delta \E[\partial_{x_\tau} H_{t+\tau} | \F_t] 
         + \Delta \E [\tilde{\E} [\partial_{\mu_\tau} \tilde{H}_{t+\tau}(X_t)] |\F_t] \bigg|^2 \bigg) dt\\
         &+  \E \int_t^T \bigg( |\Delta Y_t|^2 + \delta^2 \bigg|\Delta \partial_x \check{H}_t +  \Delta \tilde{\E}[ \partial_{\mu} \tilde{\check{H}}_t(\check{X}_t)] + \Delta \E[\partial_{x_\tau} \check{H}_{t+\tau} | \F_t] + \Delta \E [\tilde{\E} [\partial_{\mu_\tau} \tilde{\check{H}}_{t+\tau}(\check{X}_t)] |\F_t] \bigg|^2 \bigg) dt\\
         &+ \E \bigg|\lambda_0 \bigg( \Delta \partial_x g_T + \Delta \tilde{\E}[\partial_\mu \tilde{g}_T(X_T)] \bigg) + \delta \bigg( \Delta \partial_x \check{g}_T + \Delta \tilde{\E}[\partial_\mu \tilde{\check{g}}_T (\check{X}_T)] \bigg) \bigg|^2
\end{aligned}
\end{equation}
Choose $\e = 96 \max\{R^2, L\}$, and from assumption (H5.1 - H5.2) and Gronwall’s inequality, we obtain a bound for $\sup_{0 \leq t \leq T} \E|\Delta Y_t|^2$; and then substitute the it back to the same inequality, we are able to obtain the bound for $\int_0^T E|Z_t|^2 dt$. By combining these two bounds, we deduce that
\begin{equation}\label{YZ}
\begin{aligned}
& \E\left[ \sup_{0 \leq t \leq T} |Y_t|^2 +  \int_0^T  |Z_t|^2  dt \right]\\
\leq & C \E\left(\sup_{0 \leq t \leq T} |\Delta X_t|^2 + \int_0^T |\Delta \alpha_t|^2 dt\right) 
+ \delta C \E \left[  \sup_{0 \leq t \leq T} \left (  |\Delta \check{X}_t|^2  +   |\Delta \check{Y}_t|^2 \right) + \int_0^T \left(  |\Delta \check{Z}_t|^2  + |\Delta \check{\alpha}_t|^2 \right)  dt \right] \\
\end{aligned}
\end{equation}
Finally, combining \eqref{X} and \eqref{YZ}, and \eqref{alpha}, we deduce
\begin{equation}
\begin{aligned}
& \E\left[ \sup_{0 \leq t \leq T}  |\Delta X_t|^2 + \sup_{0 \leq t \leq T}  |\Delta Y_t|^2 +  \int_0^T  \left(|\Delta Z_t^2| + |\Delta \alpha_t|^2 \right) dt \right] \\
\leq & \delta C \E \left[  \sup_{0 \leq t \leq T}  |\Delta \check{X}_t|^2 + \sup_{0 \leq t \leq T}   |\Delta \check{Y}_t|^2 + \int_0^T \left ( |\Delta \check{Z}_t|^2 + |\Delta \check{\alpha}_t|^2 \right) dt \right]
\end{aligned}
\end{equation}
Let $\delta_0 = \frac{1}{2C}$, it is clear that the mapping $\Phi$ is a contraction for all $\delta \in (0, \delta_0)$. It follows that there is a unique fixed point  which is the solution of MV-FABSDE($\lambda$) for $\lambda = \lambda_0 + \delta, $ \  $\delta \in (0, \delta_0)$.

\end{proof}

\begin{proof}[Proof of Theorem \ref{existence}]
For $\lambda = 0$, FABSDE($0$) has a unique solution. Using Lemma \ref{lemma}, there exists a $\delta_0 > 0$ such that FBSDE($\delta$) has a unique solution for $\delta \in [0, \delta_0]$, assuming $(n-1)\delta_0 < 1 \leq n \delta_0$. Following by a induction argument, we repeat Lemma \ref{lemma} for $n$ times, which gives us the existence of the unique solution of FABSDE(1).
\end{proof}

\bibliographystyle{plain}
\bibliography{master}

\end{document}